\documentclass[11pt,reqno]{amsart}
\usepackage{algorithm,algorithmicx}
\usepackage{algpseudocode}
\usepackage{bbm}
\usepackage[mathlines,pagewise,left]{lineno}
\usepackage{amssymb}
\usepackage{color}
\usepackage{amsmath}
\usepackage{bcol}
\usepackage{multirow}
\usepackage[numbers]{natbib}

\usepackage{amsthm}
\usepackage{pdflscape}
\newtheorem{proposition}{Proposition}
\newtheorem{lemma}{Lemma}
\newtheorem{claim}{Claim}
\newtheorem{theorem}{Theorem}

\theoremstyle{definition}

\theoremstyle{remark}
\newtheorem{remark}{Remark}
\newtheorem{example}{Example}

\usepackage{graphicx}
 \usepackage{mathptmx}      % use Times fonts if available on your TeX system
\usepackage{algorithm,algorithmicx}
\usepackage{algpseudocode}
\usepackage{bbm}
\usepackage[mathlines,pagewise,left]{lineno}
\usepackage{amssymb}
\usepackage{color}
\usepackage{amsmath}
\usepackage{bcol}
\usepackage{multirow}
\usepackage[numbers]{natbib}
\usepackage{pdflscape}
\usepackage{units}
\usepackage[colorinlistoftodos,prependcaption,textsize=small]{todonotes}

\newcommand{\R}{\ensuremath{\mathbb{R}}}
\newcommand{\Z}{\ensuremath{\mathbb{Z}}}

\newcommand{\conv}{\ensuremath{\text{conv}}}
\newcommand{\diag}{\ensuremath{\text{diag}}}

\newcommand{\rev}[1]{{#1}}
\begin{document}

	\title[quadratic optimization with M-matrices and semi-continuous variables]
	{Strong formulations for quadratic optimization with M-matrices and semi-continuous variables%\thanks{Grants or other notes
	}
	
	
	\author{Alper Atamt{\"u}rk        \and
		Andr{\'e}s G{\'o}mez %etc.
	}
	
	\thanks{ \noindent \hskip -5mm
		A. Atamt\"urk: Department of Industrial Engineering \& Operations Research, University of California, Berkeley, CA 94720.
		\texttt{atamturk@berkeley.edu}   \\
		A. G\'{o}mez: Department of Industrial Engineering,  Swanson School of Engineering, University of Pittsburgh, Pittsburgh, PA 15261. \texttt{agomez@pitt.edu}
	}

	\maketitle

	\BCOLReport{18.01}%{Mathematical Programming}

\begin{abstract}
We study quadratic optimization with \rev{indicator} variables and an M-matrix, i.e., a PSD matrix with non-positive
 off-diagonal entries, \rev{which} arises \rev{directly} in image segmentation and portfolio optimization \rev{with transaction costs}, as well as a substructure of general quadratic optimization problems.  We prove, under mild assumptions, that the minimization problem is solvable in polynomial time by showing its equivalence to a submodular minimization problem. To strengthen the formulation, 
we decompose the quadratic function into a sum of simple quadratic functions with at most two \rev{indicator} variables each, and provide the convex-hull descriptions of these sets. 
We also describe strong conic quadratic valid inequalities.
Preliminary computational experiments indicate that the proposed inequalities can substantially improve the strength of the continuous relaxations with respect to the standard perspective reformulation.


\end{abstract}

\begin{center}
	January 2018; April 2018
\end{center}

\section{Introduction}
\rev{
Consider the quadratic optimization problem with indicator variables
\[
\text{(QOI)} \ \ \ \min \bigg \{ a'x + b'y + y'Ay \ : \ (x,y) \in C, \ 0 \le y \le x, \ x \in \{0,1\}^N \bigg \},
\]
where  $N=\{1,\ldots,n\}$, $a$ and $b$ are $n$-vectors, $A$ is an $n \times n$ symmetric matrix and $C \subseteq \R^{N \times N}$. 
Binary variables $x$ indicate a selected subset of $N$ and are often used to model non-convexities such as cardinality constraints and fixed charges.
(QOI) arises in linear regression with best subset selection \citep{Bertsimas2016}, control \citep{Gao2011}, filter design \citep{Wei2013} problems, and portfolio optimization \cite{Bienstock1996}, among others.	
In this paper, we give strong convex relaxations for the related mixed-integer set
\[
S=\big \{(x,y,t)\in \{0,1\}^N\times \R^{N}\times \R:y'Qy\leq t,\; 0 \le y_i\leq x_i \text{ for all }i\in N\big\},
\]
where $Q$ is an M-matrix \citep{plemmons1977m}, i.e., $Q\succeq 0$ and $Q_{ij}\leq 0$ if $i\neq j$. 
M-matrices arise in the analysis of Markov chains \cite{markov-m}. 
Convex quadratic programming with an M-matrix is also studied on its own right \cite{qp-m}.
Quadratic minimization with an M-matrix arises directly in a variety of applications including portfolio optimization \rev{with transaction costs} \citep{Lobo2007} and image segmentation \citep{Hochbaum2013}. 
 
There are numerous approaches in the literature for deriving strong formulations for (QOI) and $S$.
\citet{DL:ipco-qp-ind} describe lifted inequalities for (QOI) from its continuous quadratic optimization counterpart 
over bounded variables.
\citet{BM:conv-noncov} give a characterization linear inequalities obtained by strengthening gradient inequalities of a convex objective function over a non-convex set. Convex relaxations of $S$ can also be constructed from the mixed-integer epigraph of the bilinear function $\sum_{i\neq j}Q_{ij}y_iy_j$. There is an increasing amount of recent work focusing on bilinear functions \cite[e.g.,][]{boland2017bounding,boland2017extended,Luedtke2012}. However, the convex hull of such functions is not fully understood even in the continuous case. More importantly, considering the bilinear functions independent from the quadratic function $\sum_{i\in N}Q_{ii}y_i^2$ may result in weaker formulations for $S$. Another approach, applicable to general mixed-integer optimization, is to derive strong formulation based on disjunctive programming \citep{balas1985disjunctive,Ceria1999,stubbs1999branch}. Specifically, if a set is defined as the disjunction of convex sets, then its convex hull can be represented in an extended formulation using perspective functions. Such extended formulations, however, require creating a copy of each variable for each disjunction, and lead to prohibitively large formulations even for small-scale instances. There is also a increasing body of work on characterizing the convex hulls in the original space of variables, but such descriptions may be highly complex even for a single disjunction, e.g., see \cite{AN:conicmir:ipco,belotti2015conic,kilincc2015two,modaresi2016intersection}. 
}

\rev{The convex hull of $S$ is well-known for a couple of special cases. When} the matrix $Q$ is diagonal, the quadratic function \rev{$y'Qy$} is separable and 
the convex hull of $S$ can be described using the \emph{perspective reformulation} \citep{ Frangioni2006}.
This perspective formulation has a compact conic quadratic representation \citep{akturk2009strong,Gunluk2010} and is \rev{by} now a standard \rev{model strengthening} technique for mixed-integer nonlinear optimization \citep{BLTW:mp-indicator,HBCO:on-off,Mahajan2017,Wu2017}. In particular, a convex quadratic function $y'Ay$ is decomposed as $y'Dy+y'Ry$, where $A=D+R$, $D, R\succeq 0$ and $D$ is diagonal
and then each diagonal term $D_{ii} y_i^2 \le t_i$, $i \in N$, is reformulated as
$y_i^2 \le t_i x_i$.
Such decomposition and strengthening of the diagonal terms are also standard for the binary restriction, where $y_i=x_i$, $i\in N$, in which case $x'Ax \Leftrightarrow \sum_{i\in N}D_{ii}x_i+x'Rx$ \citep[e.g.][]{anstreicher2012convex,poljak1995convex}. 
The binary restriction of $S$, where $y_i=x_i$ and $Q_{ij}\leq 0$, \rev{$i \neq j$,} is also well-understood, since in that case the quadratic function $x'Qx$ is submodular \citep{Nemhauser1978} and min $\{a' x + x'Qx: x \in \{0,1\}^n\}$ is a minimum cut problem \rev{\citep{ivuanescu1965,picard1975minimum}} and, \rev{therefore}, is solvable in poynomial time. 

Whereas the set $S$ with an M-matrix is interesting on its own, the convexification 
results on $S$ can also be used to strengthen a general quadratic $y'Ay$ by decomposing $A$ as $A=Q+R$, where $Q$ is an M-matrix, \rev{and then applying the convexification results in this paper only on the $y'Qy$ term with negative off-diagonal coefficients}, generalizing the perspective reformulation approach above. \rev{We demonstrate this approach for portfolio optimization problems with negative as well as positive correlations through computations that indicate significant additional strengthening over the perspective formulation through exploiting the negative correlations.}
 
The key idea for deriving strong formulations for $S$ is decompose
the quadratic function in the definition of $S$ as 
 the sum of quadratic functions involving one or two variables:
\begin{equation}
\label{eq:quadraticDecomposition}
y'Qy=\sum_{i=1}^n\left(\sum_{j=1}^nQ_{ij}\right)y_i^2-\sum_{i=1}^n\sum_{j=i+1}^n Q_{ij}(y_i-y_j)^2.
\end{equation}  
Since a \rev{univariate} quadratic function with \rev{an indicator} is well-understood, we turn our attention to studying the mixed-integer set with two \rev{ continuous and two indicator} variables:
\begin{equation*}
X=\left\{(x,y,t)\in \{0,1\}^2\times \R^2\times \R: (y_1-y_2)^2\leq t,\; 0 \le y_i\leq x_i, \ i=1,2\right\}.
\end{equation*}

\rev{\citet{FGH:2x2decomp} also construct strong formulations for (QOI) based on $2\times 2$ decompositions. In particular, they characterize quadratic functions that can be decomposed as the sum of convex quadratic functions with at most two variables. They utilize the disjunctive convex extended formulation for the mixed-integer quadratic set 
$$\hat{X}=\left\{(x,y,t)\in \{0,1\}^2\times \R^2\times \R: q(y)\leq t,\; 0 \le y_i\leq x_i, \ i=1,2\right\},$$
where $q(y)$ is a general convex quadratic function. The authors report that the formulations are weaker when the matrix $A$ is an M-matrix, and remark on the high computational burden of solving the convex relaxations due the large number of additional variables. Additionally, \citet{Jeon2017} give conic quadratic valid inequalities for $\hat{X}$, which can be easily projected into the original space of variables, and demonstrate their effectiveness via computations. However, a convex hull description of $\hat{X}$ in the original space of variable is unknown.}

In this paper, we improve upon \rev{previous} results for the sets $S$ and $X$. In particular, our main contributions are 
($i$) showing, under mild assumptions, that the minimization of a quadratic function with an M-matrix and \rev{indicator} variables is equivalent to a submodular minimization problem and, hence, solvable in polynomial time; 
($ii$) \rev{giving} the convex hull description of $X$ \rev{in the original space of variables --- the resulting formulations for $S$ are at least as strong as the ones used by Frangioni et al. and require substantially fewer variables};
($iii$) \rev{proposing} conic quadratic inequalities amenable to use with conic quadratic MIP solvers --- the proposed inequalities dominate the ones given by Jeon et al.; 
($iv$) \rev{demonstrating} the \rev{strength and performance of the resulting formulations for (QOI)}.

\vskip 1mm
\noindent
\textit{Outline}
The rest of the paper is organized as follows. In Section~\ref{sec:preliminaries} we review the previous results for $S$ and $X$. In Section~\ref{sec:convexHullUnbounded} we study the relaxations of $S$ and $X$, where the \rev{constraints $0 \le y_i\leq x_i$ are relaxed to $y_i(1-x_i)=0$,} and the related optimization problem.
In Section~\ref{sec:convexHullBounded} we give the convex hull description of $X$. The convex hulls obtained in Sections~\ref{sec:convexHullUnbounded} and \ref{sec:convexHullBounded} cannot be immediately implemented with off-the-shelf solvers \rev{in the original space of variables}. Thus, in Section~\ref{sec:valid} we propose valid conic quadratic inequalities and discuss their strength. 
In Section~\ref{sec:extensions} we give extensions to quadratic functions with positive off-diagonal entries and continuous variables unrestricted in sign.
In Section~\ref{sec:computations} we provide a summary computational experiments and in Section~\ref{sec:conclusions} we conclude the paper.

\paragraph{Notation}Throughout the paper, we use the following convention for division by $0$: $\nicefrac{0}{0}=0$ and $\nicefrac{a}{0}=\infty$ if $a>0$. In particular, the function $p:[0,1]\times \R_+\to \R_+$ given by $p(x,y)=\nicefrac{y^2}{x}$ is the closure of the perspective function of the quadratic function $q(y)=y^2$, and is convex \citep[e.g.][p. 160]{Hiriart2013}. For a set $X\subseteq \R^N$, $\text{conv}(X)$ denotes the convex hull of $X$. Throughout, $Q$ denotes an $n\times n$ M-matrix, i.e., $Q \succeq 0$ and $Q_{ij}\leq 0$ for $i\neq j$.

\section{Preliminaries}
\label{sec:preliminaries}

In this section we briefly review the relevant results on the binary restriction of $S$ and the previous results on set $X$.

\subsection{The binary restriction of $S$}
\label{subsec:binary}

Let $S_B$ be the binary restriction of $S$, i.e. $y=x \in \{0,1\}^n$. In this case,
the decomposition
\begin{align} \label{eq:bin-decomp}
x'Qx = \sum_{i=1}^n\left(\sum_{j=1}^nQ_{ij}\right)x_i^2-\sum_{i=1}^n\sum_{j=i+1}^n Q_{ij}(x_i-x_j)^2 \le t
\end{align}
leads to $\conv(S_B)$, by simply taking the convex hull of each term.
Indeed, the quadratic problem $\min \big \{x'Qx: x \in\{0,1\}^n \big \}$ is equivalent to an undirected min-cut problem 
\cite[e.g.][]{picard1975minimum} and can be formulated as
\[
\min \sum_{i=1}^n\left(\sum_{j=1}^nQ_{ij}\right)x_i - \sum_{i=1}^n\sum_{j=i+1}^n Q_{ij} t_{ij}: x_i - x_j \le t_{ij}, \ x_j - x_i \le t_{ij}, \ 0 \le x \le 1.
\]
Decomposition \eqref{eq:bin-decomp} leading to a simple convex hull description of $S_B$ in the binary case is \rev{our} main motivation for studying decomposition \eqref{eq:quadraticDecomposition} with the \rev{indicator} variables.

\subsection{Previous results for set $X$}

Here we review the valid inequalities of Jeon et al. \cite{Jeon2017} for $X$.
Although their construction is not directly applicable 
\rev{as they assume a strictly convex function}, 
one can utilize it to obtain 
limiting inequalities. For $q(y)=y'Ay$ the inequalities of Jeon et al. are described via the inverse of the Cholesky factor of $A$. However, for $X$, we have $q(y)=(y_1-y_2)^2$ or $q(y)=y'Ay$, where $A=\left [\begin{smallmatrix} 1 & -1 \\ -1 & 1\end{smallmatrix} \right ]$ is a singular matrix and the Cholesky factor is not invertible.

However, if the matrix is given by $A= \left [\begin{smallmatrix} d_1 & -1 \\ -1 & d_2\end{smallmatrix} \right ]$ with $d_1,d_2> 1$, then their approach yields three valid inequalities:
\begin{align*}
d_2\frac{y_2^2}{x_2}-\frac{1}{d_1}x_1+\left(\frac{d_1d_2-1}{d_1}\right)\frac{y_2^2}{x_2}\leq t\\
(d_2-1)\frac{y_2^2}{x_2}+d_1\frac{y_1^2}{x_1}+\frac{x_2}{d_1}-2x_2\leq t\\
\left(\frac{d_1d_2-1}{d_1}\right)\frac{y_2^2}{x_2}+\frac{\left(\sqrt{d_1}y_1-\sqrt{\frac{1}{d_1}}y_2\right)^2}{x_1+x_2}\leq t.
\end{align*}
As $d_1, d_2 \rightarrow 1$, we arrive at three limiting valid inequalities for $X$.
\begin{proposition}%[Jeon at al. \cite{Jeon2017}]
	\label{prop:validJeon}
	The following convex inequalities are valid for $X$:
	\begin{align}
	\frac{y_2^2}{x_2}-x_1&\leq t,\label{eq:jeff1}\\
	\frac{y_1^2}{x_1}-x_2&\leq t, \label{eq:jeff2}\\
	\frac{\left(y_1-y_2\right)^2}{x_1+x_2}&\leq t. \label{eq:jeff3}
	\end{align}  
\end{proposition}
\rev{For completeness, we verify here the validity of the limiting inequalities directly. The validity of inequality \eqref{eq:jeff1} is easy to see: observe that $\nicefrac{y_2^2}{x_2}\leq 1$ for $(x,y)\in X$; then, for $x_1=0$, \eqref{eq:jeff1} reduces to the perspective formulation for the quadratic constraint $y_2^2\leq t$, and for $x_1=1$ we have $\nicefrac{y_2^2}{x_2}-x_1\leq 0 \leq t$. The validity of inequality \eqref{eq:jeff2} is proven identically. Finally, inequality \eqref{eq:jeff3} is valid since it forces $y_1=y_2$ when $x_1=x_2=0$, and is dominated by the original inequality $(y_1-y_2)^2\leq t$ for other integer values of $x$.}

Inequalities \eqref{eq:jeff1}--\eqref{eq:jeff3} are not sufficient to describe conv($X$) \rev{though}. In the next two sections we describe conv($X$) and give new conic quadratic valid inequalities dominating \eqref{eq:jeff1}--\eqref{eq:jeff3} for $X$.

\section{The unbounded relaxation}
\label{sec:convexHullUnbounded}
In this section we study the unbounded relaxations of $S$ and $X$ 
obtained by dropping the upper bound on the continuous variables:
\begin{align*}
S_U&=\left\{(x,y,t)\in \{0, 1\}^N\times \R_+^{N}\times \R:y'Qy\leq t,\;y_i(1-x_i)=0 \text{ for all }i\in N\right\},\\
X_U&=\left\{(x,y,t)\in \{0,1\}^2\times \R_+^2\times \R: (y_1-y_2)^2\leq t: y_i(1-x_i)=0,\; i=1,2\right\}.\end{align*}
In Section~\ref{sec:optimizationUnbounded} we show that the minimization of a linear function over $S_U$ is equivalent to a submodular minimization problem \rev{and, consequently, solvable in polynomial time}.
In Section~\ref{sec:inequalitiesUnbounded}, we describe $\conv(X_U)$ and in Section~\ref{sec:validUnbounded} we use the results in Section~\ref{sec:inequalitiesUnbounded} to derive valid inequalities for $S_U$.

\subsection{Optimization over $S_U$}
\label{sec:optimizationUnbounded}

We now show that the optimization of a linear function over $S_U$ can be solved in polynomial time under a mild assumption on the objective function. 
Consider the problem 
\begin{align*}
\text{(P)} \ \ \ \min \left \{ a'x+b'y+t: (x,y,t) \in S_U \right  \}, 
\end{align*}
where $Q$ is a positive definite M-matrix and $b\leq 0$. We show that (P) is  a submodular minimization problem.
The positive definiteness assumption on $Q$ ensures that an optimal solution exists. Otherwise, if there is $y \ge 0$ with $y'Qy = 0$, the problem may be unbounded.
The assumption $b\leq 0$ is satisfied in most applications (e.g., see Sections~\ref{subsec:dualNetwork} and \ref{subsec:dense}). If $b > 0$, then $y=0$ in any optimal solution.

\begin{proposition}[Characterization 15 \cite{plemmons1977m}]
	A positive definite M-matrix $Q$ is \emph{inverse-positive}, i.e., its inverse satisfies $Q_{ij}^{-1}\geq 0$ for all $i,j$. 
\end{proposition}

\begin{proposition}
	Problem (P) \rev{is equivalent to a submodular minimization problem and it is, therefore, solvable} in polynomial time.
	% if $b\leq 0$ \todo{Should we delete $b\leq 0$? It is already stated above.}.
\end{proposition}

\begin{proof}
	We assume that $a\geq 0$ (otherwise $x=1$ in any optimal solution) and that an optimal solution exists. Given an optimal solution $(x^*,y^*)$ to (P), let $T=\left\{i\in N: y_i^*>0\right\}$, $b_T$ the subvector of $b$ induced by $T$, and by $Q_T$ the submatrix of $Q$ induced by $T$. Then, from KKT conditions, we find
	$
	b_T+2Q_T y_T=0 \Leftrightarrow y_T=-\nicefrac{Q_T^{-1}b_T}{2} \cdot
	$
	Thus, an optimal solution satisfies
	$b'y^*+{y^*}'Qy^*=-\frac{b_T'Q_T^{-1}b_T}{4} \cdot $
	
	\rev{Consequently,} defining $\theta_{ij}:2^N\to \R$  for $i,j\in N$  as 
    $
	\theta_{ij}(T)= (Q_T^{-1})_{ij} \text{ if }i, j\in T \text{ and }
	0  \text{ o.w.,}
	$ observe that (P) is equivalent to the binary minimization problem 
	$$\min_{T\subseteq N} \ \ a(T)-\frac{1}{4}\sum_{i\in N}\sum_{j\in N}b_ib_j \theta_{ij}(T) \cdot$$
	
	Note that since $Q_T$ is a positive definite $M$-matrix for any $T\subseteq N$, $Q_T= \mu I_T-P_T$, where $P_T$ is a nonnegative matrix and the largest eigenvalue of $P_T$ is less than $\mu$. 
	By scaling, we may assume that $\mu=1$. % (for all $i\in N$). 
	Moreover,  $Q_T^{-1}=(I-P_T)^{-1}=\sum_{\ell=0}^\infty P_T^{\ell}$ \cite[e.g.][]{Young81}. 
	For $\ell\in \mathbb{Z}_+$ and all $i,j\in N$ let $ 
	\bar \theta_{ij}^\ell(T)=(P_T^\ell)_{ij} \text{ if }i,j\in T, \text{ and }
	0  \text{ o.w.}
	$
	Note that $\theta_{ij}(T)=\sum_{\ell=0}^\infty \bar \theta_{ij}^\ell(T)$. Finally, define for $k\in N$ and $T\subseteq N\setminus\{k\}$ the \rev{increment} function $\rho_{ij}^\ell(k,T)= \bar \theta_{ij}^\ell(T\cup\{k\})-\bar \theta_{ij}^\ell(T)$.
	\begin{claim}
		For all $i,j\in N$ and $\ell \in \Z_+$, $\bar \theta_{ij}^\ell$ is a monotone 
		supermodular function. 
	\end{claim}
	\begin{proof}
		The claim is proved by induction on $\ell$. 
		
		$\bullet$ Base case, $\ell=0$: Let $k\in N$ and $T\subseteq N\setminus \{k\}$. Note that $P_T^0=I_T$. Thus $\rho_{kk}^0(k,T)=1$, and $\rho_{ij}^0(k,T)=0$ for all cases except $i=j=k$. Thus, the marginal contributions are constant and $\bar \theta_{ij}^0$ is supermodular. Monotonicity can be checked easily. 
		
		$\bullet$ Induction step: Suppose $\bar \theta_{ij}^\ell$ is supermodular and monotone for all $i,j\in N$. 
		Observe that $\bar \theta_{ij}^{\ell+1}(T)=\sum_{t\in N}\bar \theta_{it}^{\ell}(T)P_{tj}$ if $i,j\in T$ and $\bar \theta_{ij}^{\ell+1}(T)=0$ otherwise. Monotonocity of $\bar \theta_{ij}^{\ell+1}$ follows immediately from the monotonicity of the functions $\bar \theta_{it}^{\ell}$. Now let $k\in N$ and $T_1\subseteq T_2\subseteq N\setminus\{k\}$. To prove  supermodularity, we check that $\rho_{ij}^{\ell+1}(k,T_2)-\rho_{ij}^{\ell+1}(k,T_1)\geq 0$ by considering all cases: 
		\begin{description}
			\item[ $k\not\in \{i,j\}$] If $\{i,j\}\subseteq T_1$ then $\rho_{ij}^{\ell+1}(k,T_2)-\rho_{ij}^{\ell+1}(k,T_1)=\sum_{t\in N}(\rho_{it}^{\ell}(k,T_2)-\rho_{it}^{\ell}(k,T_1))P_{tj}\geq 0$ by supermodularity of functions $\bar \theta_{it}^\ell$; if $\{i,j\}\not\subseteq T_1$ and $\{i,j\}\subseteq T_2$ then $\rho_{ij}^{\ell+1}(k,T_2)-\rho_{ij}^{\ell+1}(k,T_1)=\rho_{ij}^{\ell+1}(k,T_2)\geq 0$ by monotonicity; finally, if $\{i,j\}\not\subseteq T_2$ then $\rho_{ij}^{\ell+1}(k,T_2)-\rho_{ij}^{\ell+1}(k,T_1)=0$.
			
			\item[$k=i$] If $j\in T_1$ then $\rho_{kj}^{\ell+1}(k,T_2)-\rho_{kj}^{\ell+1}(k,T_1)=\sum_{t\in N}(\rho_{kt}^{\ell}(k,T_2)-\rho_{kt}^{\ell}(k,T_1))P_{tj}\geq 0$ by supermodularity of functions $\bar \theta_{kt}^\ell$; if $j\not\in T_1$ and $j\in T_2$ then $\rho_{kj}^{\ell+1}(k,T_2)-\rho_{kj}^{\ell+1}(k,T_1)=\bar\theta_{kj}^{\ell+1}(T_2\cup\{k\})\geq 0$; finally, if $j\not\in T_2$ then $\rho_{kj}^{\ell+1}(k,T_2)-\rho_{kj}^{\ell+1}(k,T_1)=0$. The case $k=j$ is identical.
		\end{description}
	\end{proof}
	As $\theta_{ij}(T)=\sum_{\ell=0}^\infty \bar \theta_{ij}^\ell(T)$ is a sum of supermodular functions, it is supermodular. Consequently, $\nicefrac{1}{4}\sum_{i\in N}\sum_{j\in N}b_ib_j \theta_{ij}(T)$ is a supermodular function and (P) is a submodular minimization problem, solvable \rev{with a strongly polynomial number of calls to a value oracle} \cite[e.g.][]{Orlin2009}. \rev{Evaluating the submodular function for a given set $T$, i.e., computing $a(T)-\nicefrac{b_T'Q_T^{-1}b_T}{4}$, requires only matrix multiplication and inversion, and can be done in strongly polynomial time. Therefore (P) is solvable in strongly polynomial time.} 
\end{proof}

\subsection{Convex hull of $X_U$}
\label{sec:inequalitiesUnbounded}

Consider the function $f:[0,1]^2\times \R_+^2\to \R_+$ defined as
\begin{equation}
\label{eq:defF}f(x,y)=\begin{cases}\frac{(y_1-y_2)^2}{x_1}& \text{if }y_1\geq y_2\\\frac{(y_2-y_1)^2}{x_2}& \text{if }y_1\leq y_2\end{cases}
\end{equation}
and the corresponding nonlinear inequality \begin{equation}
\label{eq:unboundedCut}
f(x,y)\leq t.
\end{equation}

\begin{remark} 
	Observe that that inequality \eqref{eq:unboundedCut} dominates inequality  \eqref{eq:jeff3} since 
	\begin{equation*}
	\frac{(y_1-y_2)^2}{x_1+x_2}\leq\frac{(y_1-y_2)^2}{\max\{x_1,x_2\}}\leq f(x,y).
	\end{equation*}
	Inequalities \eqref{eq:jeff1}--\eqref{eq:jeff2} are not valid for the unbounded relaxation \rev{as the conditions $\nicefrac{y_i^2}{x_i}\leq 1$ are not satisfied by all feasible points in $X_U$. For example, feasible points with $x_1=x_2=1$, $y_1=y_2>1$ and $t=0$ are cut off by \eqref{eq:jeff1}--\eqref{eq:jeff2}.}
\end{remark}

\begin{proposition}
	Inequality \eqref{eq:unboundedCut} 
	is valid for $X_U$.
\end{proposition}
\begin{proof}
	There are four cases to consider. If $x_1=x_2=1$, then $f(x,y)$ reduces to the original quadratic inequality $(y_1-y_2)^2$, thus the inequality is valid. If $x_1=x_2=0$, then the points in $X_U$ satisfy $y_1=y_2=0$ and $t\geq 0$; since $f(0,0)=0$, none of these points are cut off by \eqref{eq:unboundedCut}. If $x_1=1$ and $x_2=0$, then $y_2=0$ in any point in $X_U$ and, in particular, $y_1\geq y_2$; thus $f(x,y)$ reduces to the original inequality. The case where $x_1=0$ and $x_2=1$ is similar. \
\end{proof}

\rev{Observe that function $f$ is a piecewise nonlinear function, where each piece is conic quadratic representable. However, the pieces are not valid outside of the region where they are defined, e.g., $(y_1-y_2)^2\leq tx_1$ is invalid when $y_2>y_1$ as it cuts off feasible points with $x_1=y_1=0$ and $y_2>0$. Thus, inequality \eqref{eq:unboundedCut} is not equivalent to the system given by $(y_1-y_2)^2\leq tx_i$, $i=1,2$. Nevertheless, as shown in Proposition~\ref{prop:convexityF} below, \eqref{eq:unboundedCut} is a convex inequality. }

\begin{proposition}
	\label{prop:convexityF}
	The function $f$ is convex on its domain. 
\end{proposition}
\begin{proof}
	Let $(\bar{x},\bar{y}),(\hat{x},\hat{y})\in [0,1]^2\times \R_+^2$ and let $(x^*,y^*)=(1-\lambda)(\bar{x},\bar{y})+\lambda (\hat{x},\hat{y})$ 
	for $0\leq \lambda\leq 1$
	be a convex combination of $(\bar{x},\bar{y})$ and $(\hat{x},\hat{y})$. We need to prove that 
	\begin{equation}
	\label{eq:convexF}
	f(x^*,y^*)\leq (1-\lambda)f(\bar{x},\bar{y}) + \lambda f(\hat{x},\hat{y}).
	\end{equation}
	If $\bar{y}_1\geq \bar{y}_2$ and $\hat{y}_1\geq \hat{y}_2$, or $\bar{y}_1\leq \bar{y}_2$ and $\hat{y}_1\leq \hat{y}_2$, inequality \eqref{eq:convexF} holds by convexity of the individual functions in the definition of $f$. Otherwise, assume, without loss of generality, that $\bar{y}_1\geq \bar{y}_2$, $\hat{y}_1\leq \hat{y}_2$, and $y_1^*\leq y_2^*$. 
	Letting $\gamma=\lambda -(1-\lambda)\frac{\bar{y}_1-\bar{y}_2}{\hat{y}_2-\hat{y}_1}$, observe that\begin{itemize}
		\item $\gamma\leq \lambda \leq 1$.
		\item $\gamma\geq 0$, which is equivalent to $y_2^*-y_1^*\geq 0$. 
		\item $y_2^*-y_1^*=\gamma(\hat{y}_2-\hat{y}_1)$.
		\item $\gamma \hat{x}_2\leq \lambda \hat{x}_2\leq x_2^*$. 
	\end{itemize}    
	Then, we find
	\begin{align*}f(x^*,y^*)=\frac{(y_2^*-y_1^*)^2}{x_2^*}\leq \frac{(y_2^*-y_1^*)^2}{\gamma \hat{x}_2}=&\gamma\frac{(\hat{y}_2-\hat{y}_1)^2}{ \hat{x}_2}
	%\leq \lambda f(\hat{x},\hat{y})
	\leq  \lambda f(\hat{x},\hat{y})+(1-\lambda)f(\bar{x},\bar{y}). \ \ \  
	\end{align*}
\end{proof}

\rev{A consequence of Proposition~\ref{prop:convexityF} is that the convex inequality \eqref{eq:unboundedCut} can be implemented (with off-the-shelf solvers) using subgradient inequalities as for a subgradient $\xi \in \partial f(\bar{x},\bar{y})$ at a given point 
	$(\bar{x},\bar{y})$, we have
	$f(\bar{x},\bar{y})+\xi'(x-\bar{x},y-\bar{y})\leq f(x,y),$ for all points $(x,y)$ in the domain of the convex function $f$. In particular, the linear cuts
	\begin{equation}
	\label{eq:subgradientCut}
	f(\bar{x},\bar{y})+\xi'(x-\bar{x},y-\bar{y})\leq t \text{ for } \xi \in \partial f(\bar{x},\bar y)
	\end{equation}
	provide an outer-approximation of $f(x,y) \le t$ at $(\bar{x},\bar{y})$ and are valid everywhere on the domain.
A subgradient $\xi$ can be found simply by taking the gradient of the relevant piece of the function at $(\bar{x},\bar{y})$. In particular, for $\bar y_1\geq \bar y_2$ and $\bar x_1>0$, a subgradient inequality is
	\begin{equation}
	\label{eq:subgradientUnbounded}-\left(\frac{\bar y_1-\bar y_2}{\bar x_1}\right)^2x_1+2\left(\frac{\bar y_1-\bar y_2}{\bar x_1}\right)(y_1-y_2)\leq t.\end{equation}
The process outlined here to find subgradient cuts \eqref{eq:subgradientCut} for $f$ can be utilized for any convex piecewise nonlinear function, and will be used for other functions in the rest of the paper. 
Convex piecewise nonlinear functions also arise in strong formulations for mixed-integer conic quadratic optimization \cite{atamturk2017polymatroid}, 
and subgradient linear cuts for such functions were recently used in the context of the pooling problem \cite{luedtke2018strong}.}

As Theorem~\ref{theo:convexHullUnbounded} below states, inequality \eqref{eq:unboundedCut} and bound constraints for the binary variables describe the convex hull of $X_U$.
\begin{theorem}[Convex hull of $X_U$]
	\label{theo:convexHullUnbounded}
	$$\text{conv}(X_U)=\left\{(x,y,t)\in [0,1]^2\times \R_+^2\times \R: f(x,y)\leq t\right\}.$$
\end{theorem}
\begin{proof}
	Consider the optimization problems
	\begin{align*}
	(P_0)\ \ \ \ \ \ \ \ \ &\min_{(x,y,t)\in X_U} a'x+b'y+ct;\\
	(P_1)\ \ \ \ \ \ \ \ \ &\min_{(x,y,t)\in [0,1]^2\times \R_+^2\times \R} a'x+b'y+ct\text{ s.t. } f(x,y)\leq t.
	\end{align*}
	To prove the result we show that for any value of $a,b,c$, either $(P_0)$ and $(P_1)$ are both unbounded, or there exists a solution integral in $x$ that is optimal for both problems.
	If $c<0$, then $(P_0)$ and $(P_1)$ are both unbounded, and if $c=0$ then $(P_1)$ corresponds to an optimization problem over an integral polyhedron and it is easily checked that $(P_0)$ and $(P_1)$ are equivalent. Thus, the interesting case is $c>0$ or, by scaling, $c=1$. 
	Note that $t=(y_1-y_2)^2$ in any optimal solution of $(P_0)$, and $t=f(x,y)$ in any optimal solution of $(P_1)$. If $b_1, b_2\geq 0$, then $y_1=y_2=0$ is optimal with
	corresponding integer $x$ optimal for both $(P_0)$ and $(P_1)$.
	% and therefore $x_i\in \{0,1\}$ in $(P_1)$ and the problems have the same solution. 
	Moreover, if $b_1+b_2<0$, then both problems are unbounded: $x_1=x_2=1$, $y_1=y_2=\lambda$ is feasible for any $\lambda > 0$ for both problems. Thus, one needs to consider only the case where $b_1+b_2 \ge 0$ and $b_1 < 0$ or $b_2 < 0$. Without loss of generality, let $b_1<0$ and $b_2>0$. 

	\vspace{1mm}
	\noindent
	\textbf{Optimal solutions of $(P_0)$}. There exists an optimal solution with $y_2=0$ (if $0<y_2 \leq y_1$, subtracting $\epsilon>0$ from both $y_1$ and $y_2$ does not increase the objective  -- and if $y_2>y_1$, then swapping the values of $y_1$ and $y_2$ reduces the objective). Thus, $y_2=0$, $x_2=0$ if $a_2\geq 0$ and $x_2=1$ otherwise, and either $x_1=y_1=0$ or $x_1=1$ and $y_1=-\frac{b_1}{2}$, which is the stationary point of $b_1 y_1 + y_1^2$. 
 
 	\vspace{1mm}
	\noindent
	\textbf{Optimal solutions of $(P_1)$}. Note that there exists an optimal solution of $(P_1)$ where at least one of the continuous variables is $0$ (if $0<y_1,y_2$, subtracting $\epsilon>0$ from both variables does not increase the objective value --- this operation does not change the relative order of $y_1$ and $y_2$). Then, we conclude that $y_2=0$ in an optimal solution (if $y_1=0$ and $y_2>0$, then setting $y_2=0$ reduces the objective value). Moreover, when $y_2=0$, then $f(x,y)=y_1^2/x_1$. Thus, in the optimal solution $y_1=-b_1x_1/2$. Substituting in the objective, we see that $(P_1)$ simplifies to 
	$
	\min_{0\leq x_1, x_2\leq 1}
	%a_1x_1+a_2x_2-\frac{b_1^2}{2}x_1+\frac{b_1^2x_1^2}{4x_1}=
	a_2x_2+ \big (a_1-b_1^2/4 \big )x_1.
	$
	For an optimal solution, $x_2=0$ if $a_2\geq 0$ and $x_2=1$ otherwise, and $x_1=0$ if $a_1-b_1^2/4\geq 0$ and $x_1=1$ otherwise. And, if $x_1=1$, then $y_1=-b_1/2$. Hence, the optimal solutions coincide. \
\end{proof}

\subsection{Valid inequalities for $S_U$}
\label{sec:validUnbounded}

\paragraph{Inequalities in an extended formulation} 
Let $\bar Q_i = \sum_{j=1}^nQ_{ij}$ and $P = \{i \in N: \bar Q_i > 0\}$ and $\bar P = N \setminus P$. Using decomposition \eqref{eq:quadraticDecomposition} and introducing
$t_{ij}$, $1 \le i \le j \le n$, one can write a convex relaxation of $S_U$ as
\begin{align*}
\sum_{i \in \bar P} \bar Q_i y_i +
\sum_{i \in P} \bar Q_i y_i^2/x_i 
- \sum_{i=1}^n \sum_{j=i+1}^n Q_{ij} t_{ij} & \le t  \\
f(x_i, x_j, y_i, y_j) & \le t_{ij},  \ \ 1 \le i \le j \le n. 
%y_i^2/x_i & \le t_i, \ \ i \le i \le n \ \ (\text{if } \bar Q_i > 0)\\
%y_i & \le t_i, \ \ i \le i \le n \ \ (\text{if } \bar Q_i \le 0),
\end{align*}

\paragraph{Inequalities in the original space of variables} 
By projecting out the auxiliary variables $t_{ij}$ one obtains valid inequalities in the original space of variables. 
\rev{By re-indexing variables if necessary,} assume
that $y_1\geq y_2\geq \ldots \geq y_n$ to obtain the \rev{convex} inequality
\begin{equation}
\label{eq:nonlinearValidUnbounded}
\sum_{i \in \bar P} \bar Q_i y_i +
\sum_{i \in P} \bar Q_i y_i^2/x_i -
\sum_{i=1}^n\sum\limits_{j=i+1}^n Q_{ij}(y_i-y_j)^2/x_i\leq t.
\end{equation} 
Observe that the nonlinear inequality \eqref{eq:nonlinearValidUnbounded} is valid only if  $y_1\geq \ldots \geq y_n$ holds. However, we can obtain linear inequalities that are valid for $S_U$ by underestimating \rev{the convex function $
	\sum_{i \in \bar P} \bar Q_i y_i +
	\sum_{i \in P} \bar Q_i y_i^2/x_i 
	- \sum_{i=1}^n \sum_{j=i+1}^n Q_{ij} f(x_i,x_j,y_i,y_j)$} \rev{by its subgradients}.
Let $(\bar{x},\bar{y})\in [0,1]^N\times \R_+^N$ be such that $\bar{y}_1\geq \ldots \geq \bar{y}_n$ and $\bar{x}>0$. Then, the \rev{subgradient} inequality 
\begin{align*}
&-\sum_{i\in P}\bar Q_i \left(\frac{\bar y_i}{\bar x_i}\right)^2 x_i+\sum_{i=1}^n \left(\sum\limits_{j=i+1}^n\frac{ Q_{ij}(\bar{y}_i-\bar{y}_j)^2}{\bar{x}_i^2}\right) x_i\\
&+2\sum_{i\in P}\bar Q_i \frac{\bar y_i}{\bar x_i} y_i+\sum_{i\in \bar P}\bar Q_i y_i+2\sum_{i=1}^n\left( 
\sum_{j=1}^{i-1}\frac{Q_{ij}(\bar{y}_j-\bar{y}_i)}{\bar{x}_j}
-\sum\limits_{j=i+1}^n\frac{ Q_{ij}(\bar{y}_i-\bar{y}_j)}{\bar{x}_i}
\right)y_i\leq t,
\end{align*}
corresponding to a first order approximation of \eqref{eq:nonlinearValidUnbounded} around $(\bar{x},\bar{y})$, is valid for $S_U$ (regardless of the ordering of the variables). 

\section{The bounded set $X$}
\label{sec:convexHullBounded}
Let $g:[0,1]^2\times \R_+^2\to \R_+$ be defined as
\begin{equation}
\label{eq:defG}
g(x,y)=\begin{cases}
\frac{(y_1-x_2)^2}{x_1-x_2}+\frac{(x_2-y_2)^2}{x_2} & \text{if }y_2\leq x_2\leq y_1  \text{ and }x_2(x_1-y_1)\leq y_2(x_1-x_2) \\
\frac{(y_2-x_1)^2}{x_2-x_1}+\frac{(x_1-y_1)^2}{x_1} & \text{if }y_1\leq x_1\leq y_2 \text{ and }x_1(x_2-y_2)\leq y_1(x_2-x_1)\\
f(x,y) & \text{otherwise,}
\end{cases}
\end{equation}
where $f$ is the function defined in \eqref{eq:defF}.
This section is devoted to proving the main result:
% of the paper, given in Theorem~\ref{theo:convexHullBounded} below.
\begin{theorem}[Convex hull of $X$]
	\label{theo:convexHullBounded}
	$$\text{conv}(X)=\left\{(x,y,t)\in [0,1]^2\times  \R_+^3: g(x,y)\leq t,\; y_i \leq x_i,\;i=1,2 \right\}.$$
\end{theorem}

\begin{remark}
	Observe that for the binary restriction $X_B$ with $y_i=x_i$, $i=1,2$, $g(x,y) \le t$ reduces to 
%	$$g(x)=\begin{cases}x_1-x_2& \text{if }x_1\geq x_2\\ x_2-x_1& \text{if }x_1\leq x_2,\end{cases}$$
	 $|x_1 - x_2| \leq t$, which together with the bound constraints describe $\conv(X_B)$. 
\end{remark}

The rest of this section is organized as follows. In Section~\ref{sec:convexHullBinary1} we give the convex hull description of the intermediate set with \rev{two continuous} variables and one \rev{indicator} variable:
$$X_1=\left\{(x,y,t)\in \{0,1\}\times \R_+^2\times \R: (y_1-y_2)^2\leq t,\; y_1\leq x,\; y_2\leq 1\right\}.$$
In Section~\ref{sec:convexHullBinary2} we use this results to prove Theorem~\ref{theo:convexHullBounded}. Finally, in Section~\ref{sec:counterexample} we give valid inequalities for $S$.  Unlike in Section~\ref{sec:convexHullUnbounded}, the convex hull proofs in this section are constructive, \rev{i.e., we show how $g$ is constructed from the mixed-binary description of $X$, instead of just verifying that $g$ does indeed result in conv$(X)$}.

\subsection{Convex hull description of $X_1$}
\label{sec:convexHullBinary1}

Let  $g_1:[0,1]\times \R_+^2\to \R_+$ be given by
$$ g_1(x,y_1, y_2)=\begin{cases}\frac{\left(y_2-x\right)^2}{1-x}+\frac{\left(x-y_1\right)^2}{x} & \text{if }x-y_1\leq x(y_2-y_1)\\
\frac{\left(y_1-y_2\right)^2}{x} & \text{if }y_2\leq y_1\\
(y_2-y_1)^2 & \text{otherwise.}
\end{cases}$$

\begin{proposition}
	\label{prop:ConvexHull1}
	$\conv(X_1)=\left\{(x,y,t)\in [0,1]\times \R_+^2\times \R: g_1(x,y_1,y_2)\leq t,\; y_1\leq x, \; y_2\leq 1 \right\}$.
\end{proposition}
\begin{proof}
	Note that a point $(x,y,t)$ belongs to $\conv(X_1)$ if and only if there exists $(\bar{x},\bar{y},\bar{t})$, $(\hat{x},\hat{y},\hat{t})$ and $0\leq \lambda \leq 1$ such that 
	\begin{align}
	&t=(1-\lambda)\bar{t}+\lambda \hat{t}\label{eq:convT}\\
	&x=(1-\lambda)\bar{x}+\lambda \hat{x}\label{eq:convX}\\
	&y_1=(1-\lambda)\bar{y}_1+\lambda \hat{y}_1\label{eq:convY}\\
	&y_2=(1-\lambda)\bar{y}_2+\lambda \hat{y}_2\label{eq:convZ}\\&\bar{x}=0,\; \hat{x}=1\label{eq:defX}\\
	&\bar{y}_1=0,\; 0\leq \hat{y}_1\leq 1\label{eq:defY}\\
	&0\leq \bar{y}_2, \ \hat{y}_2\leq 1\label{eq:defZ}\\
	&\bar{t}\geq \bar{y}_2^2\\
	&\hat{t}\geq (\hat{y}_1-\hat{y}_2)^2.\label{eq:defT2}
	\end{align} 
	\rev{The non-convex system \eqref{eq:convT}--\eqref{eq:defT2} follows directly from the definition of the convex hull. Note that a convex extended formulation of conv($X_1$) could also be obtained using the approach proposed by \citet{Ceria1999}. See also \citet{V:cayley} for a recent approach to eliminate the auxiliary variables using Cayley embedding.  
	We now show how to project out the additional variables  $(\bar{x},\bar{y},\bar{t})$, $(\hat{x},\hat{y},\hat{t})$ to find conv$(X_1)$ in the original space of variables, which can be done directly from the non-convex formulation above. }
	
	From constraints  \eqref{eq:convX} and \eqref{eq:defX} we see $\lambda =x$, from constraint \eqref{eq:convY}  $\hat{y}_1=\frac{y_1}{x}$, from \eqref{eq:defY} $y_1\leq x$, from \eqref{eq:convZ} we find $\bar{y}_2=\frac{y_2-x\hat{y}_2}{1-x}$, and from \eqref{eq:defZ} we get $0\leq \hat{y}_2\leq 1$ and $0\leq \frac{y_2-x\hat{y}_2}{1-x}\leq 1$. Thus, \eqref{eq:convT}--\eqref{eq:defT2} is feasible if and only if $0\leq y_1 \leq x$, $0\leq y_2 \leq 1$ and there exists $\hat{y}_2$ such that 
	\begin{align*}
	&t\geq\frac{\left(y_2-x\hat{y}_2\right)^2}{1-x}+\frac{\left(x\hat{y}_2-y_1\right)^2}{x}, \ \
	0\leq \hat{y}_2\leq 1, 
	\ \ 
	\frac{y_2}{x}-\frac{1-x}{x}\leq \hat{y}_2\leq \frac{y_2}{x} \cdot
	\end{align*}
	The existence of such $\hat{y}_2$ can be checked by solving the convex optimization problem
	\begin{align*}
	\text{(M1)} \ \ \ \ \ \ \ \ \ \min\; &\varphi(\hat{y}_2):= \frac{\left(y_2-x\hat{y}_2\right)^2}{1-x}+\frac{\left(x\hat{y}_2-y_1\right)^2}{x}\\
	\text{s.t.}\;&\max\left\{0,\frac{y_2}{x}-\frac{1-x}{x} \right\}\leq \hat{y}_2\leq  \min\left\{1, \frac{y_2}{x}\right\}.
	\end{align*}
	The equation $\varphi'(\hat{y}_2)=0$ yields
	\begin{align*} &-\frac{\left(y_2-x\hat{y}_2\right)}{1-x}+\frac{\left(x\hat{y}_2-y_1\right)}{x}=0\\
	\Leftrightarrow & \hat{y}_2=y_2+y_1\frac{1-x}{x}:=\eta(x,y).
	\end{align*}
	Let $\hat{y}_2^*$ be an optimal solution to (M1). Note that $\hat{y}_2^*> 0$ whenever $ \eta(x,y)> 0$. Moreover,  $\eta(x,y)\leq \frac{y_2}{x}-\frac{1-x}{x}\implies y_1+1\leq y_2$, which can only happen if $y_1=0$ and $y_2=1$, in which case $\frac{y_2}{x}-\frac{1-x}{x}=1$ . Thus, we may assume that $\hat{y}_2^*$ is not equal to one of its lower bounds. 
	
	Now observe that $\frac{y_2}{x}\leq \eta(x,y)\Leftrightarrow y_2\leq y_1$, in which case $\eta(x,y)\leq \frac{y_1}{x}\leq 1$. Additionally, if $1\leq \eta(x,y)$, then $x\leq y_2$ and in particular $y_1\leq y_2$. Therefore, the cases $ \eta(x,y) \leq \min\{1,\frac{y_2}{x}\}$, $\eta(x,y)\geq 1$, and $\eta(x,y)\geq \frac{y_2}{x}$ are mutually exclusive if $\frac{y_2}{x}\neq x$, and the optimal solution of (M1) corresponds to setting $\hat{y}_2^*=\eta(x,y)$, $\hat{y}_2^*=1$, or $\hat{y}_2^*=\frac{y_2}{x}$, respectively. By calculating the objective function of (M1) with the appropriate value of $\hat{y}_2^*$, we find $\varphi(\hat{y}_2^*) = g_1(x,y_1,y_2)$.
	Hence, $(x,y,t)\in \conv(X_1)$ if and only if $t\geq g_1(x,y_1,y_2)$ and $0\leq y_1\leq x\leq 1$, $0\leq y_2\leq 1$.\
\end{proof}

\subsection{Convex hull description of $X$}
\label{sec:convexHullBinary2}

We use a similar argument as in the proof of Proposition~\ref{prop:ConvexHull1} to prove Theorem~\ref{theo:convexHullBounded}. 
Let $(x,y,t)$ be a point such that $0\leq y_i\leq x_i\leq 1$ 
and \emph{we additionally assume that $y_1\geq y_2$}. A point $(x,y,t)$ belongs to $\conv(X)$ if and only if there exists $(\bar{x},\bar{y},\bar{t})$, $(\hat{x},\hat{y},\hat{t})$, and $0\leq \lambda \leq 1$ such that
\begin{align}
&t=(1-\lambda)\bar{t}+\lambda \hat{t}\label{eq:convT1}\\
&x_1=(1-\lambda)\bar{x}_1+\lambda \hat{x}_1\label{eq:convX1}\\
&x_2=(1-\lambda)\bar{x}_2+\lambda \hat{x}_2\label{eq:convW1}\\
&y_1=(1-\lambda)\bar{y}_1+\lambda \hat{y}_1\label{eq:convY1}\\
&y_2=(1-\lambda)\bar{y}_2+\lambda \hat{y}_2\label{eq:convZ1}\\
&\bar{x}_2=0,\; \hat{x}_2=1\label{eq:defW1}\\
&\bar{y}_2=0,\; 0\leq \hat{y}_2\leq 1\label{eq:defZ1}\\
&0\leq \bar{y}_1\leq \bar{x}_1\leq 1, \; 0\leq \hat{y}_1\leq \hat{x}_1\leq 1\label{eq:defYX}\\
&\bar{t}\geq \bar{y}_1^2/\bar{x}_1\\
&\hat{t}\geq g_1(\hat{x}_1,\hat{y}_1,\hat{y}_2).\label{eq:defT22}
\end{align}

	\rev{The system \eqref{eq:convT1}--\eqref{eq:defT22} corresponds to $\conv(K_0\cup K_1)$, where $K_0=\{(x,y,t)\in [0,1]^2\times \R_+^2\times \R:\nicefrac{y_1^2}{x_1}\leq t,\; y_2=x_2=0\}$ and $K_1=\{(x,y,t)\in [0,1]^2\times \R_+^2\times \R:g_1(x_1,y_1,y_2)\leq t,\; x_2=1\}$. Observe that $K_0$ and $K_1$ are the convex hulls of the restrictions of $X$, where $x_2=0$ and $x_2=1$, respectively.}

Using a similar reasoning as in the proof of Proposition \ref{prop:ConvexHull1}, we find $\lambda=x_2$, $\hat{y}_2=\frac{y_2}{x_2}$, $\bar{x}_1=\frac{x_1-x_2\hat{x}_1}{1-x_2}$, $\bar{y}_1=\frac{y_1-x_2\hat{y}_1}{1-x_2}$, and 
\begin{align}
\text{(M2)} \ \ \ \ \ \ \ \ \ t\geq \min_{\hat{x}_1,\hat{y}_1}\;&\psi(\hat{x}_1,\hat{y}_1)\notag\\
\text{s.t.}\;& 0\leq \hat{y}_1\leq \hat{x}_1\leq 1\label{eq:constraints1}\\
&\hat{y}_1\leq \frac{y_1}{x_2},\; \hat{x}_1-\hat{y}_1\leq\frac{x_1-y_1}{x_2},\; \frac{x_1}{x_2}-\frac{1-x_2}{x_2}\leq \hat{x}_1, \label{eq:constraints2}
\end{align}
where
$$\psi(\hat{x}_1,\hat{y}_1):=\frac{\left(y_1-x_2\hat{y}_1\right)^2}{x_1-x_2\hat{x}_1}+
x_2 g_1(\hat x_1, \hat y_1, y_2/x_2) \cdot
$$

Thus, to find the convex hull of $X$, we need to compute in closed form the solutions of the optimization problem (M2).

\begin{lemma}
	\label{lem:functionPsi}
	There exists an optimal solution $(\hat{x}_1^*,\hat{y}_1^*)$ to (M2) such that $\hat{y}_1^*\geq \frac{y_2}{x_2}$.	
\end{lemma}
\begin{proof}
	Note that if $\hat{y}_1< \frac{y_2}{x_2}$, the function $\psi$ is non-increasing in $\hat{y}_1$ for any value of $\hat{x}_1$. Thus there exists an optimal solution where $\hat{y}_1$ is set to one of its upper bounds, i.e., either $\hat{y}_1^*=\nicefrac{y_1}{x_2}$ or $\hat{y}_1^*=\hat{x}_1^*$. Since we assume $y_1\geq y_2$ and $\hat{y}_1< \nicefrac{y_2}{x_2}$, the case $\hat{y}_1^*=\nicefrac{y_1}{x_2}$ is not possible. 
	
	Now suppose that $\hat{y}_1=\hat{x}_1$. Then observe that $1\leq \frac{y_2}{x_2} + \hat{y}_1\frac{1-\hat{x}_1}{\hat{x}_1}\Leftrightarrow \hat{x}_1\leq \frac{y_2}{x_2}$. Thus
	$$\psi(\hat{x}_1)=\frac{\left(y_1-x_2\hat{x}_1\right)^2}{x_1-x_2\hat{x}_1}+\frac{\left(y_2-x_2\hat{x}_1\right)^2}{x_2-x_2\hat{x}_1}$$
	in this case (substituting $\hat{y}_1=\hat{x}_1$). Taking the derivative, we find 
	\begin{align*}
	\psi'(\hat{x}_1)
	&=x_2\frac{y_1-x_2\hat{x}_1}{(x_1-x_2\hat{x}_1)^2}\left(-2x_1+x_2\hat{x}_1+y_1\right)+x_2\frac{(y_2-x_2\hat{x}_1)}{(x_2-x_2\hat{x}_1)^2}\left(-2x_2+x_2\hat{x}_1+y_2\right) \cdot
	\end{align*}
	Note that $y_1-x_2\hat{x}_1\geq 0$ since $\hat{x}_1=\hat{y}_1\leq \nicefrac{y_1}{x_2}$ in any feasible solution, and $y_2-x_2\hat{x}_1\geq 0$, by assumption. Additionally
	\begin{itemize}
		\item since $y_1\leq x_1$ and $\hat{x}_1=\hat{y}_1\leq \nicefrac{y_1}{x_2}\leq \nicefrac{x_1}{x_2}$, we find that $-2x_1+x_2\hat{x}_1+y_1\leq 0$,
		\item since $y_2\leq x_2$ and $\hat{x}_1\leq 1$, we find that $-2x_2+x_2\hat{x}_1+y_2\leq 0$.
	\end{itemize}  
	Therefore, $\psi'(x_1)$ is non-positive, i.e., $\psi$ is non-increasing. Then, 
	 increasing $\hat{y}_1=\hat{x}_1$ another optimal solution can be found. In particular,  an optimal solution with $\hat{y}_1^*\geq \nicefrac{y_2}{x_2}$ exits.\
\end{proof}

From Lemma \ref{lem:functionPsi} we can assume, without loss of generality, that 
\begin{equation}
\label{eq:psiForm}
\psi(\hat{x}_1,\hat{y}_1)=\frac{(y_1-x_2\hat{y}_1)^2}{x_1-x_2\hat{x}_1}+\frac{(x_2\hat{y}_1-y_2)^2}{x_2\hat{x}_1} \cdot
\end{equation}

Taking partial derivatives, we find that 
\begin{align*}
\frac{\partial \psi}{\partial \hat{y}_1}(\hat{x}_1,\hat{y}_1)=& \ 2x_2\left(-\frac{y_1-x_2\hat{y}_1}{x_1-x_2\hat{x}_1}+\frac{x_2\hat{y}_1-y_2}{x_2\hat{x}_1}\right),\\
\frac{\partial \psi}{\partial \hat{x}_1}(\hat{x}_1,\hat{y}_1)=& \
x_2 \left(\frac{y_1-x_2\hat{y}_1}{x_1-x_2\hat{x}_1}\right)^2- x_2 \left(\frac{x_2\hat{y}_1-y_2}{x_2\hat{x}_1}\right)^2.
\end{align*}
Lemmas~\ref{lem:case1}--\ref{lem:case3} characterize the optimal solutions of (M2), depending on the values of $(x,y)$. Note that if \begin{equation}\label{eq:optSufficient}
\hat{y}_1=\frac{y_2}{x_2}+\frac{\hat{x}_1}{x_1}(y_1-y_2),
\end{equation} 
then $\frac{\partial \psi}{\partial \hat{y}_1}(\hat{x}_1,\hat{y}_1)=\frac{\partial \psi}{\partial \hat{x}_1}(\hat{x}_1,\hat{y}_1)=0$, independently of the values of $\hat{x}_1$ and $\hat{y}_1$. Thus, any feasible point that satisfies \eqref{eq:optSufficient} is an optimal solution of (M2), as is the case for Lemmas~\ref{lem:case1} and
\ref{lem:case2}. In contrast, under the conditions of Lemma \ref{lem:case3}, no feasible point satisfies \eqref{eq:optSufficient} as it would violate upper bound constraints. 

\begin{lemma}
	\label{lem:case1}
	If $x_1\leq x_2$ then  $\hat{x}_1^*=\frac{x_1-\epsilon}{x_2}$, where $\epsilon>0$ is a sufficiently small number, and $\hat{y}_1^*=\frac{y_2}{x_2}+\frac{\hat{x}_1^*}{x_1}(y_1-y_2)$ is an optimal solution to $(M2)$ with objective $\psi(\hat{x}_2^*,\hat{y}_2^*)=\frac{(y_1-y_2)^2}{x_1} \cdot$
\end{lemma}
\begin{proof}
	We have $\frac{\partial \psi}{\partial \hat{y}_1}(\hat{x}_1^*,\hat{y}_1^*)=\frac{\partial \psi}{\partial \hat{x}_1}(\hat{x}_1^*,\hat{y}_1^*)=0$ and $(x_1^*,y_1^*)$ satisfies all constraints \eqref{eq:constraints1}--\eqref{eq:constraints2}. Thus, $(x_1^*,y_1^*)$ is a KKT point and, by convexity, is an optimal solution. Substituting in \eqref{eq:psiForm}, we get the result.\ 
\end{proof}

\begin{lemma}
	\label{lem:case2}
	If $x_1> x_2$ and $y_2(x_1-x_2)+y_1x_2\leq x_2x_1$, then $\hat{x}_1^*=1$ and $\hat{y}_1^*=\frac{y_2}{x_2}+\frac{\hat{x}_1^*}{x_1}(y_1-y_2)$ is an optimal solution to $(M2)$ with objective $\psi(\hat{x}_2^*,\hat{y}_2^*)=\frac{(y_1-y_2)^2}{x_1} \cdot$
\end{lemma}
\begin{proof}
	Observe that $(\hat{x}_1^*,\hat{y}_1^*)$ is feasible as
	%	\begin{align*}
	$
	\hat{y}_1^*=\frac{y_2}{x_2}+\frac{y_1-y_2}{x_1}\leq \frac{y_2}{x_2}+\frac{y_1-y_2}{x_2}=\frac{y_1}{x_2};
	\hat{y}_1^*=\frac{y_2}{x_2}+\frac{y_1-y_2}{x_1}=\frac{y_2x_1+y_1x_2-y_2x_2}{x_1x_2}\leq 1=\hat{x}_1^*;
	\hat{x}_1^*-\hat{y}_1^*= 1-\frac{y_2}{x_2}-\frac{y_1-y_2}{x_1}\leq 1-\frac{y_2}{x_1}-\frac{y_1-y_2}{x_1}=\frac{x_1-y_1}{x_1}\leq \frac{x_1-y_1}{x_2};
	\frac{x_1}{x_2}-\frac{1-x_2}{x_2}=\frac{x_1-1}{x_2}+1\leq 1= \hat{x}_1^*.
	$	
	%	\end{align*}
	Additionally, note that $\frac{\partial \psi}{\partial \hat{y}_1}(\hat{x}_1^*,\hat{y}_1^*)=\frac{\partial \psi}{\partial \hat{x}_1}(\hat{x}_1^*,\hat{y}_1^*)=0$.  Thus, $(x_1^*,y_1^*)$ is a KKT point and, by convexity, is an optimal solution. Substituting in \eqref{eq:psiForm}, we find the result.\
\end{proof}

\begin{lemma}
	\label{lem:case3}
	If $x_1> x_2$ and $y_2(x_1-x_2)+y_1x_2\geq x_2x_1$, then $\hat{x}_1^*=1$ and $\hat{y}_1^*=1$ is an optimal solution to $(M2)$ with objective $\psi(\hat{x}_2^*,\hat{y}_2^*)=\frac{(y_1-x_2)^2}{x_1-x_2}+\frac{(x_2-y_2)^2}{x_2} \cdot$
\end{lemma}
\begin{proof}
	Note that since $x_2\geq y_2$ and $y_2(x_1-x_2)+y_1x_2\geq x_2x_1$, we have $x_2(x_1-x_2)+y_1x_2\geq x_2x_1\Leftrightarrow y_1\geq x_2$ and, in particular, $\hat{y}_1^*\leq \frac{y_1}{x_2}$. Additionally, it is easily checked that all other constraints \eqref{eq:constraints1}--\eqref{eq:constraints2} are satisfied.  From  $y_2(x_1-x_2)+y_1x_2\geq x_2x_1$ we find that $\frac{x_2-y_2}{x_2}\leq\frac{y_1-x_2}{x_1-x_2}$.
	Now let $\mu_1$ and $\mu_2$ be the dual variables associated with constraints $\hat{y}_1\leq \hat{x}_1$ and $\hat{x}_1\leq 1$, respectively. Since both constraints are satisfied at equality at $(\hat{x}_1^*,\hat{y}_1^*)$, then we see that the dual variables $\mu_1$ and $\mu_2$ may take positive values without violating complementary slackness. In particular, let $\mu_1^*=2x_2\left(\frac{y_1-x_2}{x_1-x_2}-\frac{x_2-y_2}{x_2}\right)\geq 0$ and $\mu_2^*=x_2\left(\frac{y_1-x_2}{x_1-x_2}-\frac{x_2-y_2}{x_2}\right)\left(\frac{x_1-y_1}{x_1-x_2}+\frac{y_2}{x_2}\right)\geq 0$. Then, 
	$
	\frac{\partial \psi}{\partial \hat{y}_1}(\hat{x}_1^*,\hat{y}_1^*)=\mu_1^* \text{ and }
	\frac{\partial \psi}{\partial \hat{x}_1}(\hat{x}_1^*,\hat{y}_1^*)=-\mu_1^*+\mu_2^*.
	$
	Thus $(\hat{x}_1^*,\hat{y}_1^*)$ corresponds to a KKT point and, by convexity, is  optimal. Substituting in \eqref{eq:psiForm} gives the result.\
\end{proof}
Note that Lemmas~\ref{lem:case1}, \ref{lem:case2} and \ref{lem:case3} cover all cases with $y_1\geq y_2$. We can now prove the main result.

\begin{proof}[Theorem~\ref{theo:convexHullBounded}]
	If $y_1\geq y_2$, the description of the convex hull follows directly from Lemmas~\ref{lem:case1}, \ref{lem:case2} and \ref{lem:case3}. If $y_1\leq y_2$, the result follows from symmetry.\
\end{proof}

\subsection{Valid inequalities for $S$}
\label{sec:counterexample}

Similar to the discussion in Section~\ref{sec:validUnbounded}, the description of $\conv(X)$ can be used to derive strong extended convex relaxations for $S$. In order to obtain \rev{(nonlinear)} inequalities in the original space of variables, we project out the auxiliary variables for a given ordering $y_1\geq\ldots\geq y_n$ of the continuous variables with additional restrictions corresponding to conditions $x_j(x_i-y_i)\leq y_j(x_i-x_j)$ in \eqref{eq:defG}. Finally, to obtain linear inequalities valid independent of the conditions, we derive the first order approximations.

Suppose $y_1\geq\ldots\geq y_n$, and $x_j(x_i-y_i)\leq y_j (x_i-x_j)$ for $j>i$, which holds, in particular, if $x=y$. By eliminating the auxiliary variables under these conditions we obtain the inequality
\begin{equation}
\label{eq:nonlinearValidBounded}
\phi(x,y)=\sum_{i \in \bar P} \bar Q_i y_i +
\sum_{i \in P} \bar Q_i y_i^2/x_i -
\sum_{i=1}^n\sum\limits_{j=i+1}^n Q_{ij}\left(\frac{(y_1-x_2)^2}{x_1-x_2}+\frac{(x_2-y_2)^2}{x_2}\right)\leq t.
\end{equation}
\rev{Inequality \eqref{eq:nonlinearValidBounded} is only valid for the particular permutation of the continuous variables and when conditions $x_j(x_i-y_i)\leq y_j (x_i-x_j)$ for $j>i$ hold. Since 
	$
	\sum_{i \in \bar P} \bar Q_i \bar y_i +
	\sum_{i \in P} \bar Q_i \bar y_i^2/\bar x_i 
	- \sum_{i=1}^n \sum_{j=i+1}^n Q_{ij} g(\bar x_i,\bar x_j,\bar y_i,\bar y_j)=\phi(\bar x, \bar y)$, we can find valid subgradient inequalities by taking gradients of the left-hand-side of \eqref{eq:nonlinearValidBounded}. }
Let  $\pi_i=Q_{ii}+2\sum_{j=i}^{i-1}Q_{ij}$ and $\alpha_i=2\sum_{j=1}^iQ_{ij}$, and recall $\bar Q_i=\sum_{j=1}^nQ_{ij}$. The partial derivatives of $\phi$ evaluated at a point $(\bar{x},\bar{y})$ where $\bar{x}=\bar{y}$ are as follows:
\begin{align*}
\frac{\partial \phi}{\partial x_i}(\bar{x},\bar{y})&=\sum_{j=i+1}^n Q_{ij}+\sum_{j=i+1}^{i-1}Q_{ij}-\bar{Q}_i=-Q_{ii}=\pi-\alpha_i,&\quad i\in P\\
\frac{\partial \phi}{\partial x_i}(\bar{x},\bar{y})&=\sum_{j=i+1}^n Q_{ij}+\sum_{j=i+1}^{i-1}Q_{ij}=\pi-\alpha_i+\bar Q_i,&\quad i\in \bar P\\
\frac{\partial \phi}{\partial y_i}(\bar{x},\bar{y})&=-2\sum_{j=i+1}^n Q_{ij}+2\bar{Q}_i=-\alpha_i,&\quad i \in P\\
\frac{\partial \phi}{\partial y_i}(\bar{x},\bar{y})&=-2\sum_{j=i+1}^n Q_{ij}+\bar{Q}_i=-\alpha_i-\bar Q_i,&\quad i \in \bar P.
\end{align*}
Thus, since $\phi(\bar{x},\bar{y})+\nabla \phi(\bar{x},\bar{y})(x-\bar{x},y-\bar{y})\leq g(x,y)\leq t$, we obtain the linear inequality
\begin{equation}
\label{eq:polymatroidX}
\sum_{i=1}^n\pi_i x_i\leq t+\sum_{i=1}^n\alpha_i(x_i-y_i)-\sum_{i\in \bar P}\bar Q_i(x_i-y_i).
\end{equation}
Observe that inequality \eqref{eq:polymatroidX} depends only on the ordering of $\bar{x}$, but not on the actual values.

\begin{remark}
	Consider the submodular function given by $q(x)=x'Qx$. The extreme points of the extended polymatroid  \citep{Edmonds1970} associated with $q$, $\Pi$, correspond to the vectors $\pi$ in inequality \eqref{eq:polymatroidX}; thus, the convex lower envelope of $q$ is described by the function $\bar{q}(x)=\max_{\pi\in \Pi}\pi'x$ \cite{L:submodular-convex}. \citet{AB:prob-nd} employ these polymatroid inequalities for the binary case. For the  \rev{mixed-integer} case, the inequality \eqref{eq:polymatroidX} is tight for the binary restriction $x=y$, and the right hand side is relaxed as the distance between $x$ and $y$ increases.
\end{remark}

\begin{remark}
The values $\alpha_i$ in inequality \eqref{eq:polymatroidX} corresponds to the value of derivative of $q(x)$ with respect to $x_i$ when $x_j=1$ for all $j\leq i$ and $x_j=0$ for $j>i$. Atamt{\"u}rk and Jeon \cite{Atamturk2017} use lifting to derive similar inequalities for another class of nonlinear functions with \rev{indicator} variables and submodular binary restriction.
\end{remark}

\section{Valid \rev{conic quadratic} inequalities for $X$}
\label{sec:valid}

The inequalities $f(x,y)\leq t$ and $g(x,y)\leq t$ derived in Sections~\ref{sec:convexHullUnbounded} and \ref{sec:convexHullBounded} 
for $X_U$ and $X$, respectively,
cannot be directly used within off-the-shelf solvers \rev{in the original space of variables} as they are piecewise functions. However, since they are convex, they can be implemented using gradient outer-approximations at differentiable points (as discussed in Sections~\ref{sec:validUnbounded} and \ref{sec:counterexample}): given a fractional point $(\bar{x},\bar{y})$ with $\bar{x}>0$ \rev{and a subgradient $\xi \in \partial g(\bar{x},\bar{y})$}, the inequality 
\begin{equation}
\label{eq:gradient}
g(\bar{x},\bar{y})+\xi'(x-\bar{x},y-\bar{y})\leq t
\end{equation}
can be used as a cutting plane to improve the continuous relaxation. However, such an approach may require adding too many inequalities \eqref{eq:gradient} to the formulation, possibly resulting in poor performance (see also Sections~\ref{subsec:dualNetwork} and \ref{subsec:dense} for additional discussion on computations). \rev{Alternatively, an extended formulation could be used \cite[e.g.,][]{Ceria1999,FGH:2x2decomp}; however, such formulations may require a prohibitively large number of variables, resulting in hard-to-solve convex formulations and poor performance in branch-and-bound algorithms.}  
Therefore, in this section we give valid conic quadratic inequalities that provide a strong approximation of $\conv(X)$ and can be readily used within conic quadratic solvers.

\subsection{Derivation of the inequalities}
Let $L_2=\left\{(x,y,t)\in X: x_2=0\right\}$ and observe that $$\conv(L_2)=\left\{(x,y,t)\in [0,1]^2\times \R_+^2\times \R:\frac{y_1^2}{x_1}\leq t,\; y_1\leq x_1,\;  x_2=y_2=0\right\}.$$ We now consider inequalities obtained by lifting the valid inequality $\frac{y_1^2}{x_1}\leq t$ for $\conv(L_2)$, i.e., inequalities of the form 
\begin{equation}
\label{eq:form}\frac{y_1^2}{x_1}+h(x_2,y_2)\leq t\end{equation} for $X$, where $h:[0,1]\times \R_+\to \R$. We additionally require the left hand side of \eqref{eq:form} to be convex, which is the case if and only if $h$ is convex. 

\begin{proposition}
	\label{prop:lifting}
	Inequality 
	\begin{equation}
	\label{eq:valid1}
	\frac{y_1^2}{x_1}+\frac{y_2^2}{x_2}-2y_2\leq t
	\end{equation}
	is valid for $X$ and is the strongest convex inequality of the form \eqref{eq:form}.
\end{proposition}
\begin{proof}
	Any valid inequality of the form \eqref{eq:form} needs to satisfy
	\begin{align*}
	h(x_2,y_2)\leq \alpha =\min\; & \left \{ (y_1-y_2)^2 -\frac{y_1^2}{x_1} \ : \
	 0\leq y_1\leq x_1,\; x_1\in \{0,1\} \right \} \cdot
	\end{align*}
	If $x_1=0$, then $\alpha=y_2^2$; else, $\alpha=-2y_1y_2+y_2^2$. Thus, $y_1=x_1=1$ is a minimizer. We also find that $h(x_2,y_2)\leq y_2^2-2y_2$ for $x_2\in \{0,1\}$. To find the strongest convex inequality, we compute $\text{conv}(W)$, where $$W=\left\{(x_2,y_2,t_2)\in \{0,1\}\times \R_+^2: y_2^2-2y_2\leq t_2,\; y_2\leq x_2\right\}.$$ Using the perspective reformulation, one sees that 
	$$\text{conv}(W)=\left\{(x_2,y_2,t_2)\in [0,1]\times \R_+^2: \frac{y_2^2}{x_2}-2y_2\leq t_2,\; y_2\leq x_2\right\},$$ and we get inequality \eqref{eq:valid1}. \
\end{proof}

By changing the lifting order, we also get that valid inequality $
\frac{y_1^2}{x_1}+\frac{y_2^2}{x_2}-2y_1\leq t
$, or, writing the inequalities more compactly, we arrive at the convex valid inequality
\begin{equation}
\label{eq:valid12}
\frac{y_1^2}{x_1}+\frac{y_2^2}{x_2}-2\min\{y_1,y_2\}\leq t.
\end{equation} 

\begin{remark}
	Observe that inequality \eqref{eq:valid12} dominates inequality \eqref{eq:jeff2} since
	\begin{align*}
	\frac{y_1^2}{x_1}-x_2=\frac{y_1^2}{x_1}-y_2-(x_2-y_2)\leq \frac{y_1^2}{x_1}-y_2-(x_2-y_2)\frac{y_2}{x_2}= \frac{y_1^2}{x_1}+\frac{y_2^2}{x_2}-2y_2.
	\end{align*}
	Similarly,  we find that \eqref{eq:valid12} dominates inequality \eqref{eq:jeff1}.
	\ignore{Similarly,  \eqref{eq:valid2} dominates inequality \eqref{eq:jeff1}.}
\end{remark}

\begin{remark} For the binary case, $y_i=x_i$, $i=1,2$,
	inequality \eqref{eq:valid12} reduces to $|x_1 - x_2| \le t$.
\end{remark}

\subsection{Strength of the inequalities}

In order to assess the strength of inequality \eqref{eq:valid12}, we consider the optimization problem 
\begin{align*}
\min\;&a_1x_1+a_2x_2+b_1y_1+b_2y_2+t\\
\text{s.t.}\;& (y_1-y_2)^2\leq t\\
\text{(SR)} \ \ \ \ \ \ \ \ \ &\frac{y_1^2}{x_1}+\frac{y_2^2}{x_2}-2\min\{y_1,y_2\}\leq t\\
& 0\leq y_1\leq x_1\leq 1\\
& 0 \leq y_2\leq x_2 \leq 1.
\end{align*}
\rev{Inequalities \eqref{eq:valid12} are not sufficient to guarantee the integrality of $x$ in the optimal solutions of (SR) for all values of $a$ and $b$, since they do not describe $\conv(X)$ (given in Section~\ref{sec:convexHullBounded}). However, we now show that optimal solutions of (SR) are indeed integral} under mild assumptions on the coefficients $a$ and $b$. First, we prove an auxiliary lemma.

\begin{lemma}
	\label{lem:bound}
	If there exists an optimal solution to (SR) with $y_i \in \{0,1\}$ for some $i\in\{1,2\}$, then there exists an optimal solution that is integral in $x$.
\end{lemma}
\begin{proof}
	If $y_1=0$, then clearly there is an optimal solution with $x_1 \in \{0,1\}$, depending on the sign of $a_1$. Moreover, (SR) reduces to 
	$\min_{0\leq y_2\leq x_2\leq 1}\left\{a_2x_2+b_2y_2+y_2^2/x_2\right\},$
	which has an optimal integral solution in $x_2$. On the other hand,
	if $y_1=x_1=1$, then (SR) reduces to 
	$\min_{0\leq y_2\leq x_2\leq 1}\left\{a_2x_2+(b_2-2)y_2+y_2^2/x_2\right\},$
	which, again, has an optimal integral solution in $x_2$.
	The case with $y_2 \in \{0,1\}$ is symmetric. \
\end{proof}

\begin{proposition}
	\label{prop:sameSign}
	 If $a_1,a_2$ have the same sign and $b_1,b_2$ have the same sign, then (SR) has an optimal solution that is integral in $x$.
\end{proposition}
\begin{proof}
	Note that if $a_1, a_2\leq 0$, then $x_1=x_2=1$ for an optimal solution of (SR). Also, if $b_1,b_2\geq 0$, then $y_1=y_2=0$ in an optimal solution of (SR), in which case $x$ is integral in extreme point solutions.
	It remains to show that if $a_1,a_2\geq 0$ and $b_1,b_2\leq 0$, then there exists an optimal solution of (SR) that is integral in $x$.
	
	Suppose that $y_1=y_2=y$ in an optimal solution. Then $(y_1-y_2)^2=0$ and $\frac{y^2}{x_1}+\frac{y^2}{x_2}-2y\leq 0$. Thus, $t=0$ and (SR) reduces to
	$$
	\min\left\{a_1x_1+a_2x_2+(b_1+b_2)y : 0\leq y\leq \min\{x_1,x_2\}\leq 1 \right \},
	$$
	which has an optimal solution integral in $x$.
	
	Now suppose, without loss of generality, there is an optimal solution with $1 > y_1>y_2>0$ (if $y_1=1$ or $y_2=0$ then by Lemma~\ref{lem:bound} the solution is integral in $x$). Then observe that, in this case, the functions $(y_1-y_2)^2$ and $y_2^2/x_2-2y_2$ are non-increasing in $y_2$. Since $b_2\leq 0$, there exists a solution where $y_2$ is at its upper bound, i.e., $y_2=x_2$. Thus problem (SR) reduces to 

\[
	\text{(SR$'$)} \ \ 
	\min \left \{ a_1x_1+b_1y_1+(a_2 \! + \!b_2)y_2+t: 
	(y_1 \!- \!y_2)^2\leq t, \frac{y_1^2}{x_1}\!-\!y_2\leq t, y_1 \! \leq\! x_1 \!\leq 1 \right \} \cdot
\]
	Let $(\lambda, \mu, \alpha, \beta)$ be the dual variables associated with the $\leq$ constraints displayed in the order above
	and consider the dual feasibility conditions of problem (SR$'$)
	\begin{align*}
	-a_1&=-\mu_1\frac{y_1^2}{x_1^2}-\alpha+\beta\\
	-b_1&=2\lambda(y_1-y_2)+2\mu\frac{y_1}{x_1}+\alpha\\
	-(a_2+b_2)&=-2\lambda(y_1-y_2)-\mu\\
	1&=\lambda+\mu\\
	0&\leq \lambda,\mu,\alpha,\beta.
	\end{align*}
	Let $(\bar{x_1},\bar{y_1},\bar{y_2},\bar{t})$ be a KKT point  with multipliers $(\bar{\lambda},\bar{\mu},\bar{\alpha}, \bar{\beta})$ and suppose that $\bar{x}_1<1$. Then observe that for small $\epsilon>0$,  $(\frac{\bar{y}_1+\epsilon}{\bar{y}_1}\bar{x_1},\bar{y_1}+\epsilon,\bar{y_2}+\epsilon,\bar{t})$ is also a KKT point with the same multipliers. In particular, by choosing $\epsilon$ so that $1=\frac{\bar{y}+\epsilon}{\bar{y}}\bar{x}$, we see that there is an optimal solution with $x_1=1$. Then, problem (SR$'$) further simplifies to 	
\[
\text{(SR$''$)}	\ \ \ \ \min\{ b_1y_1+(a_2+b_2)y_2+t:
	  (y_1-y_2)^2\leq t,
	y_1^2-y_2\leq t \} \cdot
	\]
	It remains to show that $y_2=x_2$ is integral. Note that 
	$$y_1^2-2y_1y_2+y_2^2=y_1^2-y_2(2y_1-1)\geq y_1^2-y_2,$$
	and, therefore, constraint $y_1^2-y_2\leq t$ is not binding when $y_1 < 1$. 
	So, (SR$''$) is equivalent to $\min b_1y_1+(a_2+b_2)y_2+(y_1-y_2)^2$. However, by 
	increasing or decreasing $y_1$ and $y_2$ by the same amount it is easy to check that there exists an optimal solution where either $y_1=1$ or $y_2=0$, and from Lemma~\ref{lem:bound}  there exists an optimal integral solution. \
\end{proof}

\rev{Proposition~\ref{prop:sameSign} provides an insight on which problems inequalities \eqref{eq:valid12} may be particularly effective: if the coefficients of the binary variables and the continuous variables have the same sign, then the relaxation induced by \eqref{eq:valid12} may be close to ideal; otherwise, using subgradient inequalities may be required to find strong formulations. In our computations, this simple rule of thumb indeed results in the best performance.}

\ignore{
As Example~\ref{ex:notConvexHull} shows, if $a_1$ and $a_2$ have different signs, then the optimal solution to (SR) may not be integral. A similar example can also be constructed if $b_1$ and $b_2$ have different signs. 
\begin{example}
	\label{ex:notConvexHull}
	Consider the optimization problem
	\begin{align*}
	\delta=\min\;& -x_1+x_2-0.5y_1 -0.6y_2+t\\
	\text{s.t.}\;&(y_1-y_2)^2\leq t\\
	&x_i\in \{0,1\}, 0\leq y_i\leq x_i,\; t\geq 0,\quad &i=1,2.
	\end{align*}
	An optimal solution is $x_1^*=x_2^*=y_1^*=y_2^*=1$ and $t^*=0$, with an objective value of $\delta^*=-1.1$. The optimal solution of the continuous relaxation is $\bar{x}_1=\bar{y}_1=1$, $\bar{x}_2=\bar{y}_2=0.80$ and $\bar{t}= 0.04$ with an objective value of $\bar{\delta}=-1.14$. If constraint $\frac{y_1^2}{x_1}+\frac{y_2^2}{x_2}-2\min\{y_1,y_2\}\leq t$ is added, then the optimal solution of the corresponding continuous relaxation is $\hat{x}_1=1$, $\hat{y}_1\approx 0.85$, $\hat{x}_2=\hat{y}_2\approx 0.70$ and $\hat{t}\approx 0.025$, with an objective value of $\hat{\delta}\approx-1.1225$. Thus we see that inequalities \eqref{eq:valid12} help strengthen the continuous relaxation but do not guarantee optimality if the coefficients of the discrete variables in the objective function are of opposite signs. Finally note that $g(\hat{x},\hat{y})\approx 0.08>\hat{t}$, where $g$ is the function defined in \eqref{eq:defG}, and we find that indeed the valid inequality $g(x,y)\leq t$ would further strengthen the continuous relaxation. 
\end{example}
}

\section{Extensions to other quadratic functions with two \rev{indicator} variables}
\label{sec:extensions}

In this paper we focus on the set $X$, i.e., a \rev{mixed-integer }set with non-negative continuous variables and non-positive off-diagonal entries in the quadratic matrix. Although an in-depth study of more general quadratic functions is outside the scope of this paper, the approach used in Section~\ref{sec:valid} can be naturally extended to other quadratic functions. We briefly discuss two such extensions.

\subsection{General quadratic functions}

Observe that a general quadratic function $y'Ay$ can be decomposed as 
\begin{equation*}
y'Ay=\sum_{i=1}^n\left(\left(A_{ii}-\sum_{j\neq i}|A_{ij}|\right)y_i^2-\sum_{j>i:A_{ij}<0} A_{ij}(y_i-y_j)^2+\sum_{j>i:A_{ij}>0} A_{ij}(y_i+y_j)^2\right).
\end{equation*} 
Thus, stronger formulations for general quadratic functions may be obtained by studying the set with two continuous \rev{and two indicator} variables and positive off-diagonal term
\begin{equation*}
X_+=\left\{(x,y,t)\in \{0,1\}^2\times \R_+^2 \times \R: (y_1+y_2)^2\leq t,\; y_i\leq x_i, \ i=1,2\right\}.
\end{equation*}

\begin{proposition}
	Inequality 
	\begin{equation}
	\label{eq:valid+}
	\frac{y_1^2}{x_1}+\frac{y_2^2}{x_2}\leq t
	\end{equation}
	is valid for $X_+$ and is the strongest among inequalities of the form \eqref{eq:form}.
\end{proposition}
The proof is analogous the the proof of Proposition~\ref{prop:lifting} as is omitted for brevity.
\rev{Although} inequality \eqref{eq:valid+} is similar in spirit to \eqref{eq:valid1}, and that it is the strongest among inequalities of the form \eqref{eq:form}, it is not as strong as \eqref{eq:valid1} for $X$. In particular, an integrality result similar to Proposition~\ref{prop:sameSign} does not hold for \eqref{eq:valid+}. 

\ignore{
	\begin{example}
		Consider the optimization problems
		\begin{align*}
		\text{$(P_+^1)$} \ \ \ \ \ \ \ \ \ \min_{(x,y,t)\in X_+}\;&0.5x_1+2x_2-1.9y_1-1.3y_2+t, \\
		\text{$(P_+^2)$} \ \ \ \ \ \ \ \ \ \min_{(x,y,t)\in X_+}\;&0.4x_1+0.4x_2-.3.7y_1-3.65y_2+t.
		\end{align*}
		%It can be shown that 
		Inequality \eqref{eq:valid+} is sufficient to get an optimal integer solution in $(P_+^1)$ %(i.e., 100\% gap improvement over the natural convex relaxation), 
		but does not cut off the fractional solution corresponding to the natural convex relaxation for $(P_+^2)$. % (i.e., 0\% gap improvement).
	\end{example}
}

\subsection{Quadratic functions with continuous variables unrestricted in sign}

Consider the set 
\begin{equation*}
X_{\pm}=\left\{(x,y,t)\in \{0,1\}^2\times \R^2\times \R: (y_1\pm y_2)^2\leq t,\; -x_i\leq y_i\leq x_i \text{ for }i=1,2\right\}.
\end{equation*}
Observe that, since the continuous variables can be positive or negative, the sign inside the quadratic expression does not matter (e.g., it can be flipped via the transformation $\bar{y}_2=-y_2$). Thus we assume, without loss of generality, that it is a minus sign.

\begin{proposition}
	\label{prop:plusminus}
	Inequality \eqref{eq:jeff2}, originally proposed by Jeon et al. \cite{Jeon2017},
	is valid for $X_\pm$ and is the strongest among inequalities of the form \eqref{eq:form}.
\end{proposition}
\begin{proof}
	Any valid inequality for $X_{\pm}$ of the form \eqref{eq:form} needs to satisfy
	\begin{align*}
	h(x_2,y_2)\leq \alpha =\min\; & \left \{ (y_1-y_2)^2 -\frac{y_1^2}{x_1} \
	: \ -x_1\leq y_1\leq x_1,\; x_1\in \{0,1\} \right \} \cdot
	\end{align*}
	If $x_1=0$, then $\alpha=y_2^2$. Else, $\alpha=-2y_1y_2+y_2^2$; in this case, the minimum is attained at $y_1^*=1$ if $y_2\geq 0$ and at $y_1^*=-1$ otherwise. Thus, we find that $h(x_2,y_2)\leq y_2^2-2|y_2|$ for $x_2\in \{0,1\}$. To find the strongest convex inequality, we compute $\text{conv}(W_{\pm})$, where $W_{\pm}=\left\{(y_2,x_2,t_2)\in \{0,1\}\times \R\times \R: y_2^2-2|y_2|\leq t_2,\; -x_2\leq y_2\leq x_2\right\}.$ The convex lower envelope corresponding to the one-dimensional non-convex function $h_1(y_2)=y_2^2-2|y_2|$ for $y_2\in [-1,1]$ is the constant function equal to $-1$. Moreover, it can be shown that 
	$$\text{conv}(W_{\pm})=\left\{(y_2,x_2,t_2)\in [0,1]\times \R_+^2: -x_2\leq t_2,\; -x_2\leq y_2\leq x_2\right\}$$ and we get the convex valid inequality $\frac{y_1^2}{x_1}-x_2\leq t$ for $X_{\pm}$. \
\end{proof}

In light of Proposition~\ref{prop:plusminus},  inequalities \eqref{eq:valid1}-\eqref{eq:valid12} can be interpreted as inequalities that additionally account for the non-negativity of the continuous variables, with respect to the valid inequalities proposed by Jeon et al. \cite{Jeon2017}. Moreover, although not explicitly considered by Jeon et al., their inequalities may be particularly effective for quadratic optimization problems with \rev{indicator variables and} continuous variables unrestricted in sign. \rev{Observe that inequalities \eqref{eq:jeff1}--\eqref{eq:jeff3} are indeed valid even if the variables are not required to be non-negative -- in contrast with the inequalities $f(x,y)\leq t$, $g(x,y)\leq t$ and \eqref{eq:valid12}, which account for the non-negativity of the variables and are only valid in that case.}

\section{Computations}
\label{sec:computations}

In this section we report a summary of computational experiments performed to test the effectiveness of the proposed inequalities in a branch-and-bound algorithm. All experiments are conducted using Gurobi 7.5 solver on a workstation with a 3.60GHz Intel\textregistered \ Xeon\textregistered \ E5-1650 CPU and 32 GB main memory with a single thread. The time limit is set to one hour and Gurobi's default settings are used (except for the parameter ``PreCrush", which is set to 1 in order to use cuts). Cuts (if used) are added only at the root node using the callback features of Gurobi, and the reported times include the time used to add cuts.

\subsection{Image segmentation with $\ell$-0 penalty}
\label{subsec:dualNetwork}

Given a finite set $N$, functions $d_i:\R\to \R_+$ for $i\in N$ and $s_{ij}:\R\to \R_+$ for $i\neq j$, consider 
\begin{align*}
(D)\ \ \ \ \ \ \ \ \ \ \min_{y\in Y}\;&\sum_{i\in N}d_i(y_i)+\sum_{i\neq j}s_{ij}(y_i-y_j),
\end{align*}
where $Y\subseteq \R_+^N$. Problem (D) arises as the Markov Random Fields (MRF) problem for image segmentation, see \cite{boykov2001fast,kolmogorov2004energy}. In the MRF context, $d_i$ are the \emph{deviation} penalty functions, used to model the cost of changing the value of a pixel from the observed value $p_i$ to $y_i$, e.g., $d_i(y_i)=c_i (p_i-y_i)^2$ with $c_i\in \R_+$; functions $s_{ij}$ are the \emph{separation} penalty functions, used to model the cost of having adjacent pixels with different values, e.g., $s_{ij}(y_i-y_j)=c_{ij}(y_i-y_j)^2$ with $c_{ij} > 0$ if pixels $i$ and $j$ are adjacent, and $s_{ij}(y_i-y_j)=0$ otherwise. Often, $Y=[0,1]^N$ or is given by a suitable discretization, i.e., $y$ is a vector of integer multiples of a parameter $\varepsilon$. We consider in our computations the case $Y=[0,1]^N$, but the proposed approach can be used with any $Y$.

Problem (D) can be cast as the nonlinear dual of the undirected minimum cost network flow problem \citep{Ahuja2004} and efficient algorithms exist when all functions are convex \cite{Hochbaum2013}. In contrast,
we consider here the case where the deviation functions involve a non-convex $\ell$-0 penalty, which is often used to induce sparsity, e.g., restricting the number of pixels that can have a color different from the background color.
%, the separation functions are convex quadratic. 
In particular, 
$d_i(y_i)=a_i\|y_i\|_0+\bar{d}_i(y_i)$ with $\bar{d}_i=c_i(p_i-y_i)^2$. Thus, the problem can be formulated as
\begin{equation}
\label{eq:unconstrained}\min\; \sum_{i\in N}a_ix_i+\sum_{i\in N}c_i(p_i-y_i)^2+\sum_{i\neq j}c_{ij}t_{ij} \text{ s. t. }(x_i,x_j,y_i,y_j,t_{ij})\in X,\; \forall i\neq j.
\end{equation}

\paragraph{\textbf{Instances}} The instances are constructed as follows. The elements of $N$ correspond to points in a $k \times k$ grid, thus $n=k^2$, and separation functions $s_{ij}$ are non-zero whenever the corresponding points are adjacent in the grid. The parameters $p_i$ for $i\in N$, and $c_{ij}$ for each pair of adjacent points $i,j\in N$ are drawn uniformly between 0 and 1. We set $a_i=c_i$, where $c_i$ is generated as follows: first we draw $\tilde{c}_i$ uniformly between $0$ and $1$ for all $i\in N$, let $C_1=\sum_{i\in N}\tilde{c}_i$ and $C_2=\sum_{i:p_i\geq 0.5}(2p_i-1)$; then we set $c_i=\tilde{c}_i\frac{C_1}{C_2}$. Instances generated with these parameters are observed to have large integrality gaps.  

\paragraph{\textbf{Formulations}} We test the following formulations for solving problem \eqref{eq:unconstrained}:
\begin{description}
	\item[\texttt{\rev{Basic}}] The \rev{natural formulation 
		\begin{equation*}
		\label{eq:natural}
		\min\;\sum_{i\in N}a_ix_i+\sum_{i\in N}c_i(p_i-y_i)^2+\sum_{i\neq j}c_{ij}(y_i-y_j)^2 \text{ s.t. }0\leq y\leq x,\; x\in \{0,1\}^N.\end{equation*}}
	\item[\texttt{Perspective}]
	\rev{The perspective reformulation implemented with rotated cone constraints
		\begin{align*}
		\sum_{i\in N}c_ip_i^2+\min\;&\sum_{i\in N}a_ix_i+\sum_{i\in N}c_i\left(-2p_iy_i+z_i\right)+\sum_{i\neq j}c_{ij}(y_i-y_j)^2\\
		 \text{ s.t.}\;&y_i^2\leq z_ix_i,\; \forall i\in N\\
		 &0\leq y\leq x,\;z\geq 0,\; x\in \{0,1\}^N.\end{align*}
	}
	\item[\texttt{Conic}] \rev{The formulation with the conic quadratic inequalities \eqref{eq:valid12}
	\begin{align*}
	\sum_{i\in N}c_ip_i^2+\min\;&\sum_{i\in N}a_ix_i+\sum_{i\in N}c_i\left(-2p_iy_i+z_i\right)+\sum_{i\neq j}c_{ij}t_{ij}\\
	\text{ s.t.}\;&y_i^2\leq z_ix_i,\; \forall i\in N\\
	&(y_i-y_j)^2\leq t_{ij},\; z_i+z_j-2y_i\leq t_{ij},\;z_i+z_j-2y_j\leq t_{ij},\; \forall i\neq j\\
	&0\leq y\leq x,\;z\geq 0,\; x\in \{0,1\}^N.\end{align*}
}
	\ignore{In addition to the perspective reformulation, the conic quadratic inequalities \eqref{eq:valid12} are also added in an extended formulation.}
\end{description}
Furthermore, we also test models \texttt{Perspective+cuts} and \texttt{Conic+cuts}, where the \rev{subgradient} inequalities \eqref{eq:gradient} are used as cutting planes to strengthen the \texttt{Pers\-pective} and \texttt{Conic} formulations, respectively. If $\bar{x}_i=0$ for some $i\in N$ then we use the first-order expansion around $\bar{x}_i=10^{-5}$ instead. 

\paragraph{\textbf{Results}} Table~\ref{tab:QPDual} shows a comparison of the performance of the algorithm for each formulation for varying grid sizes. Each row in the table represents the average for five instances for a grid size. Table~\ref{tab:QPDual} displays the initial gap (\texttt{igap}), the root gap improvement (\texttt{rimp}), the number of branch and bound nodes (\texttt{nodes}), the elapsed time in seconds (\texttt{time}), and the end gap at termination (\texttt{egap}) (in brackets, we report the number of instances solved to optimality within the time limit). The initial gap is computed as $\texttt{igap}=\frac{\texttt{obj}_{\texttt{best}}-\texttt{obj}_{\texttt{cont}}}{\left|\texttt{obj}_{\texttt{best}}\right|}\rev{\times 100}$, where $\texttt{obj}_{\texttt{best}}$ is the objective value of the best feasible solution found and 
$\texttt{obj}_{\texttt{cont}}$ is the objective \rev{of the continuous relaxation of \texttt{Basic}}. The root improvement is computed as $\texttt{rimp}=
\frac{\texttt{obj}_{\texttt{relax}}-\texttt{obj}_{\texttt{cont}}}
{\texttt{obj}_{\texttt{best}}-\texttt{obj}_{\texttt{cont}}}\rev{\times 100}$, where $\texttt{obj}_{\texttt{relax}}$ is the objective value of the relaxation obtained after processing the first node of the branch-and-bound tree for a given formulation, \rev{obtained by querying Gurobi's attribute ``ObjBound" at the root node using a callback}.

We observe that the \texttt{Basic} formulation requires a substantial amount of branching before proving optimality, resulting in long solution times. The \texttt{Perspective} formulation results in a root gap improvement close to 50\% and better times and end gaps than the \texttt{Basic} formulation. However, even with the \texttt{Perspective} formulation, instances with $k \times k=400$ and larger cannot be solved to optimality leaving end gaps 15.3\% or more. In contrast, formulation \texttt{Conic} results in root gap improvements close to 100\%, and the performance of the branch-and-bound algorithm is orders-of-magnitude better than with the \texttt{Basic} and \texttt{Perspective} formulations: instances with $k \times k=400$ that are not close to being solved after one hour of computation with \texttt{Basic} and \texttt{Perspective} are solved to optimality in one second; while formulation \texttt{Basic} is able to solve in five minutes instances with $100$ variables, formulation \texttt{Conic} is able to solve in the same amount of time formulations with $2,500$ variables, i.e., instances 250 times larger.

\begin{table}[h!]
	
	\setlength{\tabcolsep}{1pt}
	\begin{center}
		\caption{Experiments with image segmentation with $\ell$-0 penalty.}
		%the nonlinear dual of the network flow problem.}
		\label{tab:QPDual}
		\scalebox{0.55}{
			\begin{tabular}{ c c c |c  r r r | c r r r | c r r r| c r r r|c r r r}
				\hline \hline
				 & \multirow{2}{*}{$k \times k$} &
				\multirow{2}{*}{\texttt{igap}} & \multicolumn{4}{c|}{\textbf{\texttt{Basic}}} & \multicolumn{4}{c|}{\textbf{\texttt{Perspective}}}& \multicolumn{4}{c|}{\textbf{\texttt{Perspective+cuts}}}& \multicolumn{4}{c|}{\textbf{\texttt{Conic}}}& \multicolumn{4}{c}{\textbf{\texttt{Conic+cuts}}}\\
				&&&&\texttt{nodes}&\texttt{time}&\texttt{egap}&\texttt{rimp}&\texttt{nodes}&\texttt{time}&\texttt{egap}&\texttt{rimp}&\texttt{nodes}&\texttt{time}&\texttt{egap}&\texttt{rimp}&\texttt{nodes}&\texttt{time}&\texttt{egap}&\texttt{rimp}&\texttt{nodes}&\texttt{time}&\texttt{egap}\\
				\hline
				 & 100 & 51.0 &  & 2,065,285 & 301 & 0.0[5] & 47.9 & 70,898 & 17 & 0.0[5]& 99.6 & 27,006 & 601 & 0.0[5] & 99.4 & 7 & 0 & 0.0[5]& 99.7 & 7 & 0 & 0.0[5] \\
				 & 400 & 47.7 &   & 9,520,774 & 3,600 & 34.0[0] & 48.6 & 5,277,876 & 3,600 & 15.3[0]& 93.2 & 305 & 2 & 0.0[5] & 99.5 & 59 & 1 & 0.0[5]& 99.5 & 58 & 1 & 0.0[5] \\
				 & 2,500 & 47.9 &  & 1,091,872 & 3,600 & 46.3[0] & 45.6 & 682,406 & 3,600 & 25.6[0]& 47.2 & 38,989 & 2,235 & 9.9[2] & 99.3 & 17,561 & 393 & 0.0[5]& 99.6 & 9,220 & 210 & 0.0[5] \\
				& 10,000 & 47.4 &  & 167,529 & 3,600 & 47.2[0] & 45.9 & 131,986 & 3,600 & 25.9[0]& 32.4 & 25,992 & 3,600 & 0.2[0] & 99.5 & 25,842 & 3,600 & 0.1[0]& 99.6 & 26,695 & 3,600 & 0.1[0] \\
				\hline\hline
			\end{tabular}
		}
	\end{center}
\end{table}

Formulation \texttt{Conic+cuts} results in very modest improvement in the strength of the continuous relaxation when compared with \texttt{Conic} (less than 0.3\% additional root gap improvement) and almost no difference in terms of nodes, times or end gaps. Observe that in \eqref{eq:unconstrained} the coefficients of the linear objective terms corresponding to the discrete and continuous variables have the same sign, and the experimental results are consistent with Proposition~\ref{prop:sameSign} --- \texttt{Conic} indeed is a very close approximation of inequalities \eqref{eq:gradient} in this case. 

Note that if cuts are added without the approximation given by inequalities \eqref{eq:valid12} (formulation \texttt{Perspective+cuts}), the root improvement is substantial for small instances but it degrades as the size increases. We conjecture that the required number of cuts to obtain an adequate relaxation increases with the size of the instances. Thus, for larger instances, Gurobi may stop adding cuts before obtaining a strong relaxation. Additionally, to solve second-order conic subproblems in branch-and-bound, solvers like Gurobi construct a linear outer approximation of the convex sets; adding a large number of cuts may interfere with the construction of the outer approximation, leading to weak relaxations of the convex set, which is observed for instances with $k\times k = 10,000$. Using the approximation of the convex hull derived in Section~\ref{sec:valid} as a starting point appears to circumvent such numerical difficulties.

\rev{Finally, we remark that for the larger instances that are not solved to optimality by \texttt{Conic}, high quality solutions and tight lower bounds are found within a few seconds, but branching is ineffective to close the remaining gap. To illustrate, Figure~\ref{fig:timeMRF} presents the time to prove an optimality gap of at most 1\%, as a function of the dimension $n$ of the problem. We see that the proposed approach scales very well (almost linearly) up to $n=20,000$. In particular, the lower bound found corresponds to the one obtained at the root node, and the feasible solutions are found within a small number (50--60) of branch-and-bound nodes. Memory limit is reached for instances with $n>20,000$.}

\begin{figure}[!h ]
	\centering
	\includegraphics[width=0.9\textwidth,trim={8cm 6cm 8cm 6cm},clip]{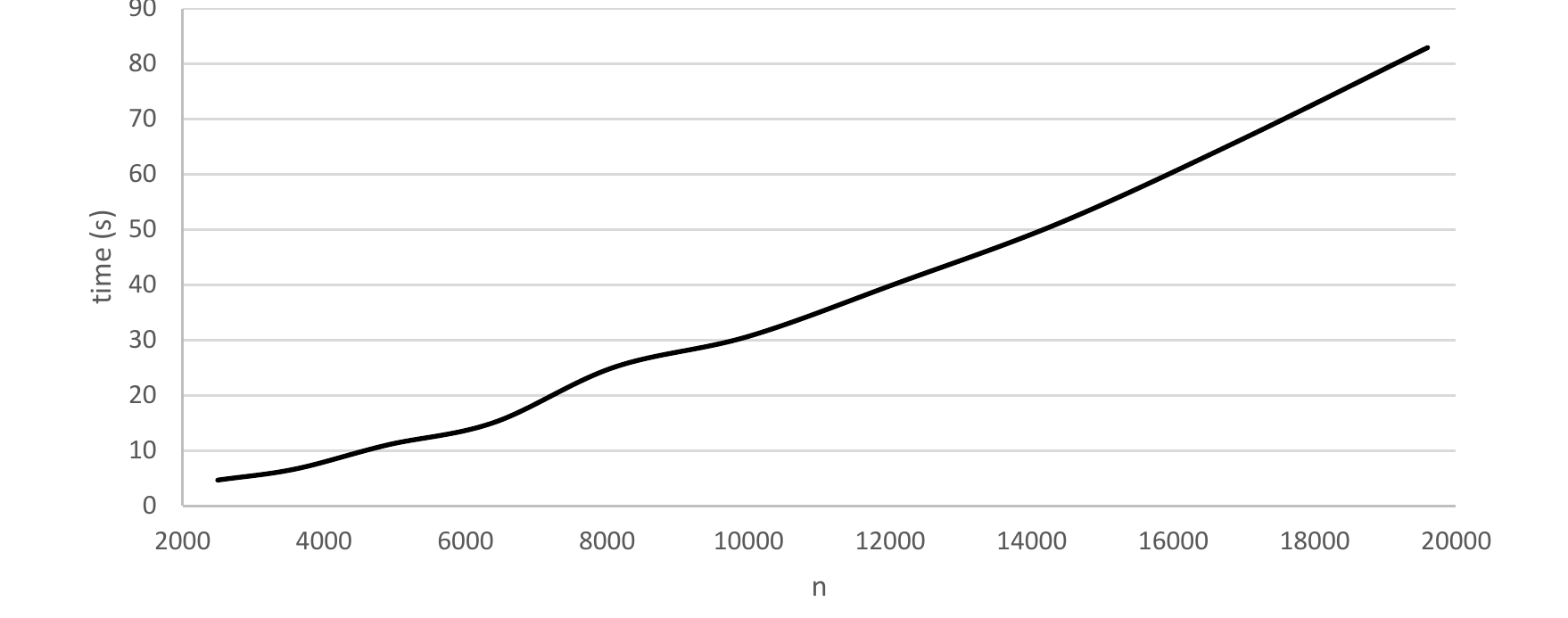}
	\caption{Time to prove an optimality gap of 1\% with \texttt{Conic} as a function of the dimension $n=k\times k$.}
	\label{fig:timeMRF}
\end{figure}

\subsection{Portfolio optimization with transaction costs}
\label{subsec:transaction} Consider a simple portfolio optimization problem with transaction costs similar to the one discussed in \cite[p.146]{cornuejols2006optimization}. However, in our case, transactions have a fixed cost and there is a restricted number of transactions. For simplicity, we \rev{first} consider assets with uncorrelated returns. \rev{In this context, an M-matrix arises directly due to the buying and selling decisions. In Section~\ref{subsec:dense} we present computations with a general covariance matrix, from which an M-matrix corresponding to the negatively correlated assets can be extracted to apply the reformulations.}

Let $N$ be the set of assets, $\mu,\sigma\in \R_+^N$ be the vectors of expected returns and standard deviations of returns. Let $w\in \R_+^N$ denote the current holdings in each asset, let $a^+,a^-\in \R_+^N$ be the fixed transaction costs associated with buying and selling any quantity, $c^+,c^-\in \R^N$ be the variable transaction costs and profits of buying and selling each asset, let $u^+,u^-\in \R_+^N$ be the upper bounds on the transactions, and let $k$ be the maximum number of transactions. Then the problem of finding a minimum risk portfolio that satisfies a given expected return $b\in \R$ with at most $k$ transactions can be formulated as the mixed-integer quadratic problem:
\begin{align*}
\min\;& v(y)=\sum_{i\in N}\sigma_i^2 (w_i+y_i^+-y_i^-)^2\\
\text{s.t.}\;& \sum_{i\in N}\left(\mu_iw_i+y_i^+(\mu_i-c_i^+)-y_i^-(\mu_i-c_i^-)-a_i^+x_i^+-a_i^-x_i^-\right)\geq b\\
&\sum_{i\in N}(x_i^++x_i^-)\leq k\\
& 0\leq y_i^+\leq u_i^+x_i^+,\; 0\leq y_i^-\leq u_i^-x_i^-,\rev{\;x_i^++x_i^-\leq 1,}\quad \forall i\in N\\
&(x^+,x^-,y^+,y^-)\in  \{0,1\}^N\times \{0,1\}^N \times \R_+^N\times \R_+^N,
\end{align*}
where $v(y)$ is the variance of the new portfolio, the decision variables $y_i^+$ ($y_i^-$) indicate the amount bought (sold) in asset $i$ and the variables $x_i^+$ ($x_i^-$) indicate whether asset $i$ is bought (sold). Note that the quadratic objective function is nonseparable and the corresponding quadratic matrix is positive semi-definite but not positive definite; therefore, the classical perspective reformulation cannot be used. Additionally, observe that the portfolio optimization problem can be reformulated by adding continuous variables $t\in \R_+^N$, constraints $(x_i^+,x_i^-,y_i^+,y_i^-,t_{i})\in X$ for all $i\in N$ to minimize the linear objective 
\begin{equation}
\label{eq:portfolioReformulation}\sum_{i\in N} \sigma_i^2(2 w_i(y_i^+-y_i^-)+t_{i}) \cdot \end{equation}
Note that since each continuous variable is involved in exactly one term in the objective, the extended formulation given by \eqref{eq:portfolioReformulation} and constraints $(x_i^+,x_i^-,y_i^+,y_i^-,t_{i})\in \conv(X)$ results in the convex envelope of $v(y)$. 

\paragraph{\textbf{Instances}} The instances are constructed as follows. We set $w_i=u_i^+=u_i^-=1$ for all $i\in N$. Coefficients $\sigma_i$ are drawn uniformly between $0$ and $1$, $\mu_i$ are drawn uniformly between $0$ and $2\sigma_i$, the transactions costs and profits $c_i^+$ and $c_i^-$ are drawn uniformly between $0$ and $\mu_i$, the fixed costs $a_i^+$ and $a_i^-$ are drawn uniformly between $0$ and $(\mu_i-c_i^+)$ and $(\mu_i-c_i^-)$, respectively. The target return is set to $\beta\sum_{i\in N}\mu_i$ where $\beta > 0$ is a parameter; $k$ is set to $n/10$.

\paragraph{\textbf{Formulations}} We test the formulations \texttt{Basic}, \texttt{Basic+cuts}, \texttt{Conic},  and \texttt{Co\-nic+cuts}, as defined in Section~\ref{subsec:dualNetwork}. As mentioned above, the perspective reformulation cannot be used for these instances. 

\paragraph{\textbf{Results}}Table~\ref{tab:QPGurobiPortfolio} shows the results for varying number of assets $n$ and values of the expected return $\beta$. Observe that instances with lower values of $\beta$ are more difficult to solve for the \texttt{Basic} formulation: low $\beta$ results in more feasible solutions, and more branch-and-bound nodes need to be explored before proving optimality. We also see that the \texttt{Basic} formulation is not effective for instances with $250$ or more assets, where most instances (27 out of 30) are not solved to optimality within the time limit and leaving large end gaps at termination. On the other hand, the other three formulations achieve root improvements of over $90\%$ in most cases, and lead to much lower solution times and end gaps.

Observe that for the portfolio problem, the coefficients of $y_i^+$ and $y_i^-$ in the objective and return constraints have opposite signs. Thus, we expect the approximation given by $\texttt{Conic}$ not to be as effective as in Section~\ref{subsec:dualNetwork} and, therefore, the cuts to have a larger impact in closing the root gaps. Indeed, we see in these experiments that adding cuts leads to an additional 2\% to 4\% root improvement (compared to the 0.3\% improvement observed in Section~\ref{subsec:dualNetwork})\footnote{The root gap improvements of 95\% achieved by \texttt{Conic} indicate that the approximation given in Section~\ref{sec:valid} is strong and considerably better than the natural continuous relaxation.}. In particular,  formulation \texttt{Basic+cuts} is able to solve all instances in seconds, even instances with low values of $\beta$ where all other formulations struggle. 

\begin{table}[h!]
	
	\setlength{\tabcolsep}{1pt}
	\begin{center}	
		\caption{Experiments with portfolio optimization with fixed transaction costs.}
		\label{tab:QPGurobiPortfolio}
		\scalebox{0.6}{
			\begin{tabular}{ c c c |c  c c c | c c c c | c c c c| c c c c  }
				\hline \hline
				\multirow{2}{*}{\texttt{$n$}} & \multirow{2}{*}{$\beta$} &
				\multirow{2}{*}{\texttt{igap}} & \multicolumn{4}{c|}{\textbf{\texttt{Basic}}} &  \multicolumn{4}{c|}{\textbf{\texttt{Basic+cuts}}}& \multicolumn{4}{c|}{\textbf{\texttt{Conic}}}& \multicolumn{4}{c}{\textbf{\texttt{Conic+cuts}}}\\
				&&&\texttt{rimp}&\texttt{nodes}&\texttt{time}&\texttt{egap}&\texttt{rimp}&\texttt{nodes}&\texttt{time}&\texttt{egap}&\texttt{rimp}&\texttt{nodes}&\texttt{time}&\texttt{egap}&\texttt{rimp}&\texttt{nodes}&\texttt{time}&\texttt{egap}\\
				\hline
				\multirow{3}{*}{100}& 0.95 & 30.8 & 0.0 & 39,963 & 11 & 0.0[5]& 98.9 & 57 & 0 & 0.0[5] & 86.9 & 822 & 4 & 0.0[5] & 92.8 & 1,069 & 4 & 0.0[5]\\
				& 0.98& 27.9 & 0.0 & 6,926 & 2 & 0.0[5]& 93.2 & 130 & 1 & 0.0[5] & 98.4 & 35 & 0 & 0.0[5] & 94.4 & 167 & 1 & 0.0[5] \\
				& 1.00& 32.7 & 0.0 & 3,229 & 1 & 0.0[5]& 97.9 & 32 & 0 & 0.0[5] & 96.9 & 49 & 0 & 0.0[5]  & 97.2 & 37 & 0 & 0.0[5]\\
				\multicolumn{3}{c|}{\textbf{Average}}&\textbf{0.0} &\textbf{ 16,706} & \textbf{5} & \textbf{0.0[15]}&\textbf{96.7} &\textbf{ 76} & \textbf{0} & \textbf{0.0[15]} &\textbf{94.1} &\textbf{ 302} & \textbf{2} & \textbf{0.0[15]}&\textbf{94.8} &\textbf{ 425} & \textbf{2} & \textbf{0.0[15]} \\
				\hline
				&&&&&&&&&&&&&&&&\\
				\multirow{3}{*}{250}& 0.95 & 32.4 & 0.0 & 5,344,016 & 3,600 & 15.0[0]& 98.8 & 176 & 0 & 0.0[5] & 94.0 & 175,859 & 2,880 & 1.7[1] & 96.0 & 233,024 & 2,880 & 1.1[1]\\
				& 0.98& 26.0 & 0.0 & 4,831,484 & 3,227 & 6.2[1] & 97.8 & 210 & 1 & 0.0[5]& 99.1 & 27 & 0 & 0.0[5] & 98.4 & 50,689 & 720 & 0.3[4] \\
				& 1.00& 29.4 & 0.0 & 4,518,960 & 2,970 & 4.0[1]& 97.3 & 2,061 & 49 & 0.0[5] & 97.4 & 3,597 & 38 & 0.0[5]  & 97.0 & 3,858 & 130 & 0.0[5]\\
				\multicolumn{3}{c|}{\textbf{Average}}&\textbf{0.0} &\textbf{4,898,153} & \textbf{3,265} & \textbf{8.4[2]}&\textbf{98.0} &\textbf{ 816} & \textbf{17} & \textbf{0.0[15]} &\textbf{96.8} &\textbf{ 59,827} & \textbf{973} & \textbf{0.6[11]}&\textbf{97.2} &\textbf{ 95,857} & \textbf{1,243} & \textbf{0.5[10]} \\
				\hline
				&&&&&&&&&&&&&&&&\\
				\multirow{3}{*}{500}& 0.95 & 32.3 & 0.0 & 2,906,338 & 3,600 & 24.5[0]& 97.6 & 387 & 2 & 0.0[5] & 95.2 & 26,640 & 1,441 & 0.6[3] & 97.2 & 139,686	 & 3,600 & 0.9[0]\\
				& 0.98& 26.1 & 0.0 & 3,096,026 & 3,600 & 16.4[0]& 98.0 & 343 & 3 & 0.0[5] & 96.4 & 295 & 2 & 0.0[5] & 99.1 & 182 & 1 & 0.0[5] \\
				& 1.00& 32.8 & 0.0 &3,076,324 & 3,600 & 18.8[0]& 97.5 & 328 & 2 & 0.0[5] & 93.4 & 330 & 2 & 0.0[5]  & 97.0 & 254 & 1 & 0.0[5]\\
				\multicolumn{3}{c|}{\textbf{Average}}&\textbf{0.0} &\textbf{ 3,026,229} & \textbf{3,600} & \textbf{19.9[0]}&\textbf{97.7} &\textbf{ 353} & \textbf{2} & \textbf{0.0[15]} &\textbf{95.0} &\textbf{ 9,088} & \textbf{481} & \textbf{0.2[13]}&\textbf{97.7} &\textbf{ 46,707} & \textbf{1,201} & \textbf{0.3[10]} \\
				\hline\hline
			\end{tabular}
		}
	\end{center}
\end{table}	

\subsection{General convex quadratic functions}
\label{subsec:dense}
The quadratic matrices used in the previous computations had specific structures, given by the applications considered. Although our results are \rev{for} M-matrices, in this section, we test the strength of the formulations for more general problems, with dense matrices having positive and negative off-diagonal entries. To employ the results developed for M-matrices, we simply apply the strengthening on the pairs of variables with a negative off-diagonal entry. Toward this end, we consider the mean-variance portfolio optimization
\begin{align*}
\min\;& y'Ay\\
\text{s.t.}\;& b'y\geq r\\
(MV)\ \ \ \ \ \ \ \ \ \ & 1'x \le  k\\
& 0\leq y\leq x \\
&x\in \{0,1\}^n.
\end{align*}	
where the objective is to minimize the portfolio variance $y'Ay$, where $A$ is a covariance matrix, subject to meeting a target return and satisfying sparsity constraints. 

\paragraph{\textbf{Instances}} 
In order test the effect of  positive off-diagonal elements and diagonal dominance, 
the matrix $A$ is constructed as follows: Let $\rho\geq 0$ be a parameter that controls the magnitude of the positive off-diagonal entries of $A$, and $\delta\geq 0$ be a parameter that controls the diagonal dominance of $A$. First, we construct a factor matrix $F=GG'$, where each entry in $G_{20\times 20}$ in drawn uniformly from $[-1,1]$, and an exposure matrix $X_{n\times 20}$ such that $X_{ij}=0$ with probability $0.8$, and $X_{ij}$ is drawn uniformly from $[0,1]$, otherwise. Then we construct an auxiliary matrix $\bar{A}=XFX'$. Then, for $i\neq j$, we set $A_{ij}=\bar{A}_{ij}$ if $\bar{A}_{ij}\leq 0$, and we set $A_{ij}=\rho \bar{A}_{ij}$ otherwise\footnote{The matrices generated this way have only 20.1\% of the off-diagonal entries negative on average -- the rest are positive if $\rho>0$ and $0$ if $\rho=0$. The ratio of the magnitude of the negative entries vs. the total, i.e., $\frac{\sum_{i\neq j: A_{ij}<0}|A_{ij}|}{\sum_{i\neq j}|A_{ij}|}$, is on average $0.72$ if $\rho=0.1$, $0.57$ if $\rho=0.2$ and $0.34$ if $\rho=0.5$.}. 
Finally, $\upsilon_i$ is drawn uniformly from $[0, \delta\bar{\sigma}]$, 
where $\bar{\sigma}=\frac{1}{n}\sum_{i\neq j}|A_{ij}|$,
and $A_{ii}=\sum_{j\in N}|A_{ij}|+\upsilon_i$. \rev{Observe that the auxiliary matrix $\bar A$ represents a low-rank matrix obtained from a 20-factor model, and $\diag(\upsilon)$ is a diagonal matrix representing the residual variances not explained by the factor model. The matrix $A$ is obtained by scaling the positive off-diagonals of $\bar A$ by $\rho$, and updating the diagonal entries to ensure positive definiteness by imposing diagonal dominance.}
Additionally, $b_i$ is drawn uniformly between $0.5U_{ii}$ and $1.5U_{ii}$. \rev{Finally, we let $r=0.25 \times \sum_{i\in N}b_i$ and $k=n/5$ for ``small" instances, and $r=0.125 \times \sum_{i\in N}b_i$ and $k=n/10$ for ``large" instances.} 

\paragraph{\textbf{Formulations}} We test the same formulations as in Section \ref{subsec:dualNetwork}. In this case, the diagonal matrix $\diag(\upsilon)$ is used for the \texttt{Perspective} formulation. \rev{In particular, formulations \texttt{Perspective+cuts}, \texttt{Conic} and \texttt{Conic+cuts} are based on the decomposition of the objective function given by
	\begin{align*}
\min\;&\sum_{i\in N}\upsilon_i z_i+\sum_{A_{ij}< 0}|A_{ij}|t_{ij}+y'(A-Q-\diag(\upsilon))y \\
\text{ s.t.}\;&y_i^2\leq z_ix_i,\;\forall i\in N, \quad (x_i,x_j,y_i,y_j,t_{ij})\in X,\; \forall i\neq j: A_{ij}< 0,\end{align*}
where $Q_{ij}=\min\{0,A_{ij}\}$ for $i\neq j$ and $Q_{ii}=-\sum_{j\neq i}Q_{ij}$. By construction, $A-Q-diag(\upsilon)$ is positive semi-definite.}
\paragraph{\textbf{Results}} 

Table \ref{tab:QPGurobiConstrained} presents the results for matrices with non-positive off diagonal entries (i.e., $\rho=0$) and varying diagonal dominance $\delta$. Table \ref{tab:QPGurobiConstrainedPositive} presents the results for matrices with fixed diagonal dominance and varying magnitudes for positive off-diagonal entries $\rho$. 
We see that, in all cases formulation \texttt{Conic} results in better root gap improvements than \texttt{Perspective} and \texttt{Basic}. The gap improvements depend on the parameters $\delta$ and $\rho$. In Table~\ref{tab:QPGurobiConstrained} we see that \texttt{Conic} formulation closes an additional 30\% to 40\% gap with respect to \texttt{Perspective} (independent of the diagonal dominance $\delta$). In Table~\ref{tab:QPGurobiConstrainedPositive} we observe that, as expected, \texttt{Conic} formulation is more effective at closing root gaps when the magnitude $\rho$ for the positive off-diagonal entries is small. Nevertheless, for all instances formulations \texttt{Conic} and \texttt{Conic+cuts} result in significantly stronger root improvements than \texttt{Perspective} (at least 15\%, and often much more) and the number of nodes required to solve the instances is decreased by at least an order of magnitude. 

	\begin{table}[h!]
	\setlength{\tabcolsep}{0.5pt}
	\begin{center}
		\caption{Experiments with non-positive off diagonal entries and varying diagonal dominance, \rev{$k=n/5$}.}
		\label{tab:QPGurobiConstrained}
		\scalebox{0.55}{
			\begin{tabular}{ c c c |c  c c c | c c c c | c c c c| c c c c|c c c c}
				\hline \hline
				\multirow{2}{*}{\texttt{$n$}} & \multirow{2}{*}{$\delta$} &
				\multirow{2}{*}{\texttt{igap}} & \multicolumn{4}{c|}{\textbf{\texttt{Basic}}} & \multicolumn{4}{c|}{\textbf{\texttt{Perspective}}}& \multicolumn{4}{c|}{\textbf{\texttt{Perspective+cuts}}}& \multicolumn{4}{c|}{\textbf{\texttt{Conic}}}& \multicolumn{4}{c}{\textbf{\texttt{Conic+cuts}}}\\
				&&&&\texttt{nodes}&\texttt{time}&\texttt{egap}&\texttt{rimp}&\texttt{nodes}&\texttt{time}&\texttt{egap}&\texttt{rimp}&\texttt{nodes}&\texttt{time}&\texttt{egap}&\texttt{rimp}&\texttt{nodes}&\texttt{time}&\texttt{egap}&\texttt{rimp}&\texttt{nodes}&\texttt{time}&\texttt{egap}\\
				\hline
				\multirow{3}{*}{60}& 0.1 & 88.2 &  & $4\cdot 10^5$ & 86 & 0.0[5] & 7.2 & $4\cdot 10^5$ & 99 & 0.0[5]& 19.2 & 15,230 & 544 & 0.0[5] & 43.6 & 3,704 & 107 & 0.0[5]& 43.9 & 4,653 & 154 & 0.0[5] \\
				& 0.5& 80.2 &  & $5\cdot 10^5$ & 103 & 0.0[5] & 28.0 & $2\cdot 10^5$ & 47 & 0.0[5]& 38.9 & 3,243 & 92 & 0.0[5] & 66.1 & 1,783 & 44 & 0.0[5]& 66.6 & 1,567 & 49 & 0.0[5] \\
				& 1.0& 74.0 &  & $6\cdot 10^5$ & 121 & 0.0[5] & 44.4 & $6\cdot 10^4$ & 18 & 0.0[5]& 52.8 & 1,335 & 35 & 0.0[5] & 81.5 & 863 & 14 & 0.0[5]& 82.3 & 709 & 19 & 0.0[5] \\
				\multicolumn{3}{c|}{\textbf{Average}}&&$\mathbf{ 5\cdot 10^5}$ & \textbf{103} & \textbf{0.0[15]} &\textbf{26.5} &$\mathbf{ 2\cdot10^5}$ & \textbf{55} & \textbf{0.0[15]}&\textbf{37.0} &\textbf{ 6,603} & \textbf{224} & \textbf{0.0[15]}&\textbf{63.7} &\textbf{ 2,117} & \textbf{55} & \textbf{0.0[15]}&\textbf{64.3} &\textbf{ 2,310} & \textbf{74} & \textbf{0.0[15]} \\
				\hline
				&&&&&&&&&&&&&&&&&&\\
				\multirow{3}{*}{80}& 0.1 & 90.3 &  & $1\cdot 10^7$ & 3,600 & 9.7[0] & 7.2 & $9\cdot 10^6$ & 3,600 & 10.1[0]& 4.0 & 31,194 & 3,600 & 16.1[0] & 37.0 & 26,657 & 2,758 & 5.7[2]& 37.3 & 36,998 & 2,776 & 4.6[2] \\
				& 0.5& 82.8 &  & $1\cdot 10^7$ & 3,600 & 10.5[0] & 28.2 & $6\cdot 10^6$ & 2,902 & 2.8[3]& 16.8 & 29,220 & 3,017 & 4.0[2] & 60.2 & 11,367 & 1,108 & 0.0[5]& 60.4 & 13,898 & 1,208 & 0.0[5] \\
				& 1.0& 77.0 &  & $1\cdot 10^7$ & 3,600 & 9.5[0] & 44.1 & $2\cdot 10^6$ & 988 & 0.0[5] & 27.2 & 4,889 & 566 & 0.0[5]&78.4 & 2,689 & 183 & 0.0[5]& 79.0 & 3,395 & 233 & 0.0[5] \\
				\multicolumn{3}{c|}{\textbf{Average}}& &$\mathbf{ 1\cdot 10^7}$ & \textbf{3,600} & \textbf{9.9[0]} &\textbf{26.5} &$\mathbf{ 5\cdot 10^6}$ & \textbf{2,496} & \textbf{4.3[8]}&\textbf{16.0} &\textbf{ 21,768} & \textbf{2,394} & \textbf{6.7[7]}&\textbf{58.5} &\textbf{ 13,571} & \textbf{1,350} & \textbf{1.9[12]}&\textbf{58.9} &\textbf{ 18,097} & \textbf{1,406} & \textbf{1.5[12]} \\
				\hline
				&&&&&&&&&&&&&&&&&&\\
				\multirow{3}{*}{100}& 0.1 & 90.2 &  & $1\cdot 10^7$ & 3,600 & 30.0[0] & 6.4 & $6\cdot 10^6$ & 3,600 & 29.3[0]& 2.8 & 14,855 & 3,600 & 35.8[0] & 37.1 & 19,660 & 3,600 & 19.6[0]& 37.0 & 17,047 & 3,600 & 21.6[2] \\
				& 0.5& 83.0 &  & $1\cdot 10^7$ & 3,600 & 27.5[0] & 25.2 & $5\cdot 10^6$ & 3,600 & 18.7[0]& 12.8 & 11,912 & 3,600 & 16.4[0] & 58.6 & 16,398 & 3,432 & 7.7[1]& 58.7 & 18,645 & 3,600 & 7.9[0] \\
				& 1.0& 77.3 &  & $1\cdot 10^7$ & 3,600 & 25.0[0] & 39.9 & $6\cdot 10^6$ & 3,600 & 10.0[0] & 19.7 & 16,144 & 3,236 & 4.8[1]&75.0 & 11,376 & 1,824 & 2.1[3]& 75.4 & 10,588 & 1,822 & 2.5[3] \\
				\multicolumn{3}{c|}{\textbf{Average}}& &$\mathbf{ 1\cdot 10^7}$ & \textbf{3,600} & \textbf{27.5[0]} &\textbf{23.8} &$\mathbf{ 6\cdot 10^6}$ & \textbf{3,600} & \textbf{19.3[0]}&\textbf{11.8} &\textbf{ 14,304} & \textbf{3,479} & \textbf{19.0[1]}&\textbf{56.9} &\textbf{ 15,811} & \textbf{2,952} & \textbf{9.8[4]}&\textbf{57.1} &\textbf{ 15,426} & \textbf{3,007} & \textbf{10.7[3]} \\
				\hline\hline
			\end{tabular}
		}
	\end{center}
\end{table}

\begin{table}[h!]
	\setlength{\tabcolsep}{0.5pt}
	\begin{center}
		\caption{Experiments with constant diagonal dominance \& varying  positive off-diagonal entries, \rev{$k=n/5$}.}
		\label{tab:QPGurobiConstrainedPositive}
		\scalebox{0.55}{
			\begin{tabular}{  c c c |c  c c c | c c c c | c c c c| c c c c|c c c c}
				\hline \hline
				\multirow{2}{*}{\texttt{$n$}} & \multirow{2}{*}{$\rho$} &
				\multirow{2}{*}{\texttt{igap}} & \multicolumn{4}{c|}{\textbf{\texttt{Basic}}} & \multicolumn{4}{c|}{\textbf{\texttt{Perspective}}}& \multicolumn{4}{c|}{\textbf{\texttt{Perspective+cuts}}}& \multicolumn{4}{c|}{\textbf{\texttt{Conic}}}& \multicolumn{4}{c}{\textbf{\texttt{Conic+cuts}}}\\
				&&&&\texttt{nodes}&\texttt{time}&\texttt{egap}&\texttt{rimp}&\texttt{nodes}&\texttt{time}&\texttt{egap}&\texttt{rimp}&\texttt{nodes}&\texttt{time}&\texttt{egap}&\texttt{rimp}&\texttt{nodes}&\texttt{time}&\texttt{egap}&\texttt{rimp}&\texttt{nodes}&\texttt{time}&\texttt{egap}\\
				\hline
				\multirow{3}{*}{60}& 0.1 & 62.4 & & $7\cdot 10^5$ & 153 & 0.0[5] & 46.0 & $7\cdot 10^4$ & 22 & 0.0[5]& 56.1 & 10,165 & 62 & 0.0[5] & 77.6 & 2,141 & 19 & 0.0[5]& 78.1 & 2,065 & 23 & 0.0[5] \\
				& 0.2& 57.3 &  & $7\cdot 10^5$ & 144 & 0.0[5] & 46.8 & $7\cdot 10^4$ & 22 & 0.0[5]& 56.4 & 16,642 & 89 & 0.0[5] & 73.5 & 3,314 & 20 & 0.0[5]& 73.9 & 3,261 & 24 & 0.0[5] \\
				& 0.5& 51.2 &  & $6\cdot 10^5$ & 128 & 0.0[5] & 48.0 & $6\cdot 10^4$ & 19 & 0.0[5]& 53.6 & 22,526 & 137 & 0.0[5] & 65.1 & 8,635 & 36 & 0.0[5]& 65.5 & 8,742 & 60 & 0.0[5] \\
				\multicolumn{3}{c|}{\textbf{Average}}& &$\mathbf{ 7\cdot 10^5}$ & \textbf{142} & \textbf{0.0[15]} &\textbf{46.9} &$\mathbf{ 6\cdot 10^5} $& \textbf{21} & \textbf{0.0[15]}&\textbf{55.4} &\textbf{ 16,444} & \textbf{96} & \textbf{0.0[15]}&\textbf{72.1} &\textbf{ 4,696} & \textbf{25} & \textbf{0.0[15]}&\textbf{72.5} &\textbf{ 4,689} & \textbf{36} & \textbf{0.0[15]} \\
				\hline
				&&&&&&&&&&&&&&&&&&\\
				\multirow{3}{*}{80}& 0.1 & 64.4 & & $1\cdot 10^7$ & 3,600 & 7.6[0] & 46.9 & $2\cdot 10^6$ & 852 & 0.0[5] & 32.8 & 53,774 & 1,401 & 0.4[4]& 77.4 & 8,979 & 244 & 0.0[5]& 78.2 & 8,551 & 183 & 0.0[5] \\
				& 0.2& 58.8 &  & $1\cdot 10^7$ & 3,600 & 5.9[0] & 48.1 & $2\cdot 10^6$ & 881 & 0.0[5]& 37.8 & 98,151 & 1,997 & 0.6[4] & 74.3 & 25,152 & 349 & 0.0[5]& 75.4 & 22,630 & 327 & 0.0[5] \\
				& 0.5& 51.8 &  & $1\cdot 10^7$ & 3,255 & 3.2[1] & 49.7 & $8\cdot 10^5$ & 391 & 0.0[5]& 43.7 & 185,839 & 2,462 & 0.4[4] &67.8 & 66,779 & 482 & 0.0[5]& 68.5 & 64,512 & 535 & 0.0[5] \\
				\multicolumn{3}{c|}{\textbf{Average}}& &$\mathbf{ 1\cdot 10^7}$ & \textbf{3,485} & \textbf{5.5[1]} &\textbf{48.2} &$\mathbf{ 1\cdot 10^6}$ & \textbf{708} & \textbf{0.0[15]}&\textbf{38.1} &\textbf{ 112,588} & \textbf{1,953} & \textbf{0.5[12]}&\textbf{73.2} &\textbf{ 33,637} & \textbf{358} & \textbf{0.0[15]}&\textbf{74.0} &\textbf{ 31,898} & \textbf{349} & \textbf{0.0[15]} \\
				\hline
				&&&&&&&&&&&&&&&&&&\\
				\multirow{3}{*}{100}& 0.1 & 65.0 &  & $9\cdot 10^6$ & 3,600 & 23.1[0] & 42.3 & $5\cdot 10^6$ & 3,600 & 9.1[0]& 28.8 & 65,628 & 3,600 & 6.4[0] & 73.0 & 83,300 & 2,667 & 2.5[2]& 73.8 & 67,074 & 2,904 & 2.6[2] \\
				& 0.2& 59.4 &  & $9\cdot 10^6$ & 3,600 & 20.9[0] & 43.9 & $5\cdot 10^6$ & 3,600 & 7.8[0]& 32.9 & 72,439 & 3,600 & 9.0[0] & 70.6 & 122,553 & 3,031 & 2.8[2]& 71.2 & 116,173 &3,033 & 3.3[1] \\
				& 0.5& 52.5 &  & $9\cdot 10^6$ & 3,600 & 17.2[0] & 46.2 & $5\cdot 10^6$ & 3,600 & 5.4[0]& 39.1 & 136,082 & 3,600 & 7.7[0] &64.4 & 261,440 & 3,327 & 3.8[1]& 64.8 & 270,701 & 3,396 & 3.7[1] \\
				\multicolumn{3}{c|}{\textbf{Average}}& &$\mathbf{ 9\cdot 10^6}$ & \textbf{3,600} & \textbf{20.4[0]} &\textbf{44.2} &$\mathbf{ 5\cdot 10^6}$ & \textbf{3,600} & \textbf{7.4[0]}&\textbf{33.6} &\textbf{ 91,383} & \textbf{3,600} & \textbf{7.4[0]}&\textbf{69.3} &\textbf{ 155,764} & \textbf{3,008} & \textbf{3.0[5]}&\textbf{69.9} &\textbf{ 151,316} & \textbf{3,111} & \textbf{3.2[4]} \\
				\hline\hline
			\end{tabular}
		}
	\end{center}
\end{table}

Observe that the stronger formulations of \texttt{Conic} and \texttt{Conic+cuts} do not necessarily lead to better solution times for small instances. 
Nevertheless, for the larger instances ($n=100$), using the \texttt{Conic} formulation leads to faster solution times, lower end gaps and more instances solved to optimality for all values of $\delta$ and $\rho$.
As in Section~\ref{subsec:dualNetwork}, we observe little difference between \texttt{Conic} and \texttt{Conic+cuts} --- consistent with Proposition~\ref{prop:sameSign}--- and that \texttt{Perspective+cuts} is not effective in closing the root gap. Approximating the nonlinear function with gradient inequalities appears to cause numerical issues as adding cuts weakens the relaxation contrary to expectations. Please see our comments at the end of Section~\ref{subsec:dualNetwork}. 

\rev{Finally, observe that the formulations tested require adding $O(n^2)$ additional variables, one for each negative off-diagonal entry in $A$. Thus, solving the continuous relaxations may be computationally expensive for large values of $n$. Table~\ref{tab:QPGurobiLarge} illustrates this point for matrices with $\rho=0$ and $\delta=1$. It shows, for the \texttt{Basic}, \texttt{Perspective} and \texttt{Conic} formulations, the value of the best feasible solution found (\texttt{sol}), the value of the lower bound after one hour of branch and bound  (\texttt{ebound}), the value of the lower bound after processing the root node (\texttt{rbound}), the time used to process the root node in seconds (\texttt{rtime}), and the number of nodes explored in one hour (\texttt{nodes}). Each row represents the average over five instances, and the values of \texttt{sol}, \texttt{ebound} and \texttt{rbound} are scaled so that the best feasible solution found for a given instance has value $100$. Observe that for $n\geq 150$ the lower bound found by \texttt{Conic} at the root node is stronger than the lower bounds found by other formulations after one hour of branch-and-bound. However, the continuous relaxations of \texttt{Conic} are difficult to solve for large values of $n$, leading to few branch-and-bound nodes explored and few or no feasible solutions found within the time limit.

	\begin{table}[h!]
	\setlength{\tabcolsep}{1pt}
	\begin{center}
		\caption{Experiments with $n\geq 100$ and $k=n/10$.}
		\label{tab:QPGurobiLarge}
		\scalebox{0.6}{
			\begin{tabular}{ c  |  c c c c c |  c c c c c|  c c c c c c c c c c c c c c}
				\hline \hline
				\multirow{2}{*}{\texttt{$n$}} & \multicolumn{5}{c|}{\textbf{\texttt{Basic}}} & \multicolumn{5}{c|}{\textbf{\texttt{Perspective}}}& \multicolumn{5}{c}{\textbf{\texttt{Conic}}}\\
				&\texttt{sol}&\texttt{ebound}&\texttt{rbound}&\texttt{rtime}&\texttt{nodes}&\texttt{sol}&\texttt{ebound}&\texttt{rbound}&\texttt{rtime}&\texttt{nodes}&\texttt{sol}&\texttt{ebound}&\texttt{rbound}&\texttt{rtime}&\texttt{nodes}\\
				\hline
				100& 100.0 & 94.9 &   11.0 & 0.09 & 12,375,694 & 100.0 & 100.0 & 42.2 & 0.05& 1,968,600 & 100.0 & 100.0 & 77.9 & 2.13 & 2,176 \\
				150& 100.0 & 61.3 &   11.7 & 0.07 & 9,739,922 & 100.3 & 81.3 & 45.6 & 0.08& 3,788,060 & 100.6 & 96.2 & 83.8 & 141.46 & 3,174 \\
				200& 100.0 & 46.5 &   12.4 & 0.11 & 6,382,960 & 100.3 & 72.5 & 48.4 & 0.13& 2,644,816 & - & 90.8 & 86.4 & 1090.73 & 1,531 \\
				250& 100.0 & 34.7 &   11.6 & 0.22 & 4,092,948 & 100.3 & 72.5 & 48.4 & 0.21& 1,692,204 & - & 82.7 & 82.7 & 1732.13 & 3 \\
				300& 100.0 & 29.5 &   12.0 & 0.41 & 2,763,780 & 100.9 & 61.0 & 47.1 & 0.32& 1,166,534 & - & 86.1 & 86.1 & 2333.81 & 1 \\
				\hline\hline
			\end{tabular}
		}
	\end{center}
\end{table}
}

\rev{A possible approach that achieves a compromise between the strength and the size of the formulation is to apply the proposed conic inequalities for a subset of the matrix: given an M-matrix Q, choose $I\subset \left\{(i,j)\in N\times N: Q_{ij}<0\right\}$ and use the formulation 
\begin{align*}
\min\;&\sum_{i\in P}\bar Q_{i} z_i+ \sum_{i \in \bar P} \bar Q_{i} y_i
-\sum_{(i,j)\in I}Q_{ij}t_{ij}-\sum_{(i,j)\not\in I}Q_{ij}(y_i-y_j)^2 \\
\text{ s.t.}\;&y_i^2\leq z_ix_i,\;\forall i\in P, \quad (x_i,x_j,y_i,y_j,t_{ij})\in X,\; \forall (i,j)\in I.\end{align*}
In particular, if $|I|\approx 4n$, then the results in Section~\ref{subsec:dualNetwork} suggest that the formulations would scale well. Additionally, the component corresponding to the remainder, $-\sum_{(i,j)\not\in I}Q_{ij}(y_i-y_j)^2$, could be further strengthened by linear inequalities \eqref{eq:polymatroidX} (and other subgradient inequalities corresponding to points where $\bar y\neq \bar x$) in the original space of variables instead of extended reformulations. An effective implementation of such a partial strengthening is beyond the scope of the current paper. }

\vspace{-2mm}

\section{Conclusions}
\label{sec:conclusions}

In this paper we show, under mild assumptions, that minimization of a quadratic function with an M-matrix \rev{with indicator} variables is a submodular minimization problem, hence, solvable in polynomial time. We derive strong formulations using the convex hull description of  
non-separable quadratic terms with two \rev{indicator} variables arising from a decomposition of the quadratic function.
Additionally, we provide strong conic quadratic valid 
inequalities approximating the convex hulls. 
The derived formulations generalize previous results in the binary case and separable case, and the inequalities dominate valid inequalities given in the literature. % for the set considered. 
Computational experiments indicate that the proposed conic formulations may be
significantly more effective
compared to the natural convex relaxation and the perspective reformulation.

\bibliographystyle{spbasic}      % basic style, author-year citations
\bibliography{Bibliography}

\begin{thebibliography}{49}
\providecommand{\natexlab}[1]{#1}
\providecommand{\url}[1]{{#1}}
\providecommand{\urlprefix}{URL }
\expandafter\ifx\csname urlstyle\endcsname\relax
  \providecommand{\doi}[1]{DOI~\discretionary{}{}{}#1}\else
  \providecommand{\doi}{DOI~\discretionary{}{}{}\begingroup
  \urlstyle{rm}\Url}\fi
\providecommand{\eprint}[2][]{\url{#2}}

\bibitem[{Ahuja et~al(2004)Ahuja, Hochbaum, and Orlin}]{Ahuja2004}
Ahuja RK, Hochbaum DS, Orlin JB (2004) A cut-based algorithm for the nonlinear
  dual of the minimum cost network flow problem. Algorithmica 39:189--208

\bibitem[{Akt{\"u}rk et~al(2009)Akt{\"u}rk, Atamt{\"u}rk, and
  G{\"u}rel}]{akturk2009strong}
Akt{\"u}rk MS, Atamt{\"u}rk A, G{\"u}rel S (2009) A strong conic quadratic
  reformulation for machine-job assignment with controllable processing times.
  Oper Res Lett 37:187--191

\bibitem[{Anstreicher(2012)}]{anstreicher2012convex}
Anstreicher KM (2012) On convex relaxations for quadratically constrained
  quadratic programming. Mathematical Programming 136:233--251

\bibitem[{Atamt{\"u}rk and Bhardwaj(2018)}]{AB:prob-nd}
Atamt{\"u}rk A, Bhardwaj A (2018) Network design with probabilistic capacities.
  Networks 71:16--30

\bibitem[{Atamt{\"u}rk and Gomez(2017)}]{atamturk2017polymatroid}
Atamt{\"u}rk A, Gomez A (2017) Submodularity in conic quadratic mixed 0-1
  optimization. arXiv preprint arXiv:170505918

\bibitem[{Atamt\"urk and Jeon(2017)}]{Atamturk2017}
Atamt\"urk A, Jeon H (2017) Lifted polymatroid for mean-risk optimization with
  indicator variables. BCOL Research Report 17.01, UC Berkeley

\bibitem[{Atamt{\"u}rk and Narayanan(2007)}]{AN:conicmir:ipco}
Atamt{\"u}rk A, Narayanan V (2007) Cuts for conic mixed integer programming.
  In: Fischetti M, Williamson DP (eds) Proceedings of the 12th International
  IPCO Conference, pp 16--29

\bibitem[{Balas(1985)}]{balas1985disjunctive}
Balas E (1985) Disjunctive programming and a hierarchy of relaxations for
  discrete optimization problems. SIAM Journal on Algebraic Discrete Methods
  6:466--486

\bibitem[{Belotti et~al(2015)Belotti, G{\'o}ez, P{\'o}lik, Ralphs, and
  Terlaky}]{belotti2015conic}
Belotti P, G{\'o}ez JC, P{\'o}lik I, Ralphs TK, Terlaky T (2015) A conic
  representation of the convex hull of disjunctive sets and conic cuts for
  integer second order cone optimization. In: Numerical Analysis and
  Optimization, Springer, pp 1--35

\bibitem[{Bertsimas et~al(2016)Bertsimas, King, and Mazumder}]{Bertsimas2016}
Bertsimas D, King A, Mazumder R (2016) Best subset selection via a modern
  optimization lens. The Annals of Statistics 44:813--852

\bibitem[{Bienstock(1996)}]{Bienstock1996}
Bienstock D (1996) Computational study of a family of mixed-integer quadratic
  programming problems. Mathematical Programming 74:121--140

\bibitem[{Bienstock and Michalka(2014)}]{BM:conv-noncov}
Bienstock D, Michalka A (2014) Cutting-planes for optimization of convex
  functions over nonconvex sets. SIAM Journal on Optimization 24:643--677

\bibitem[{Boland et~al(2017{\natexlab{a}})Boland, Dey, Kalinowski, Molinaro,
  and Rigterink}]{boland2017bounding}
Boland N, Dey SS, Kalinowski T, Molinaro M, Rigterink F (2017{\natexlab{a}})
  Bounding the gap between the {McCormick} relaxation and the convex hull for
  bilinear functions. Mathematical Programming 162(1-2):523--535

\bibitem[{Boland et~al(2017{\natexlab{b}})Boland, Gupte, Kalinowski, Rigterink,
  and Waterer}]{boland2017extended}
Boland N, Gupte A, Kalinowski T, Rigterink F, Waterer H (2017{\natexlab{b}})
  Extended formulations for convex hulls of graphs of bilinear functions. arXiv
  preprint arXiv:170204813

\bibitem[{Bonami et~al(2015)Bonami, Lodi, Tramontani, and
  Wiese}]{BLTW:mp-indicator}
Bonami P, Lodi A, Tramontani A, Wiese S (2015) On mathematical programming with
  indicator constraints. Mathematical Programming 151:191--223

\bibitem[{Boykov et~al(2001)Boykov, Veksler, and Zabih}]{boykov2001fast}
Boykov Y, Veksler O, Zabih R (2001) Fast approximate energy minimization via
  graph cuts. IEEE Transactions on Pattern Analysis and Machine Intelligence
  23:1222--1239

\bibitem[{Ceria and Soares(1999)}]{Ceria1999}
Ceria S, Soares J (1999) Convex programming for disjunctive convex
  optimization. Mathematical Programming 86:595--614

\bibitem[{Cornuejols and
  T{\"u}t{\"u}nc{\"u}(2006)}]{cornuejols2006optimization}
Cornuejols G, T{\"u}t{\"u}nc{\"u} R (2006) Optimization Methods in Finance,
  vol~5. Cambridge University Press

\bibitem[{Dong and Linderoth(2013)}]{DL:ipco-qp-ind}
Dong H, Linderoth J (2013) On valid inequalities for quadratic programming with
  continuous variables and binary indicators. In: Goemans M, Correa J (eds)
  Proc. IPCO 2013, Springer, Berlin, pp 169--180

\bibitem[{Edmonds(1970)}]{Edmonds1970}
Edmonds J (1970) Submodular functions, matroids, and certain polyhedra. In: Guy
  R, Hanani H, Sauer N, Sch\"onenheim J (eds) Combinatorial Structures and
  Their Applications, Gordon and Breach, pp 69--87

\bibitem[{Frangioni and Gentile(2006)}]{Frangioni2006}
Frangioni A, Gentile C (2006) Perspective cuts for a class of convex 0--1 mixed
  integer programs. Mathematical Programming 106:225--236

\bibitem[{Frangioni et~al(2016)Frangioni, Gentile, and
  Hungerford}]{FGH:2x2decomp}
Frangioni A, Gentile C, Hungerford J (2016) Decompositions of semidefinite
  matrices and the perspective reformulation of nonseparable quadratic
  programs. Report R-16-10, IASI, Rome

\bibitem[{Gao and Li(2011)}]{Gao2011}
Gao J, Li D (2011) Cardinality constrained linear-quadratic optimal control.
  IEEE Transactions on Automatic Control 56:1936--1941

\bibitem[{G{\"u}nl{\"u}k and Linderoth(2010)}]{Gunluk2010}
G{\"u}nl{\"u}k O, Linderoth J (2010) Perspective reformulations of mixed
  integer nonlinear programs with indicator variables. Mathematical Programming
  124:183--205

\bibitem[{Hijazi et~al(2012)Hijazi, Bonami, Cornu\'{e}jols, and
  Ouorou}]{HBCO:on-off}
Hijazi H, Bonami P, Cornu\'{e}jols G, Ouorou A (2012) Mixed-integer nonlinear
  programs featuring ``on/off" constraints. Computational Optimization and
  Applications 52:537--558

\bibitem[{Hiriart-Urruty and Lemar{\'e}chal(2013)}]{Hiriart2013}
Hiriart-Urruty JB, Lemar{\'e}chal C (2013) Convex Analysis and Minimization
  Algorithms I: Fundamentals, vol 305. Springer Science \& Business Media

\bibitem[{Hochbaum(2013)}]{Hochbaum2013}
Hochbaum DS (2013) Multi-label markov random fields as an efficient and
  effective tool for image segmentation, total variations and regularization.
  Numerical Mathematics: Theory, Methods and Applications 6:169--198

\bibitem[{Iv{\u{a}}nescu(1965)}]{ivuanescu1965}
Iv{\u{a}}nescu PL (1965) Some network flow problems solved with pseudo-boolean
  programming. Operations Research 13:388--399

\bibitem[{Jeon et~al(2017)Jeon, Linderoth, and Miller}]{Jeon2017}
Jeon H, Linderoth J, Miller A (2017) Quadratic cone cutting surfaces for
  quadratic programs with on--off constraints. Discrete Optimization 24:32--50

\bibitem[{Keilson and Styan(1973)}]{markov-m}
Keilson J, Styan GPH (1973) Markov chains and {M}-matrices: Inequalities and
  equalities. Journal of Mathematical Analysis and Applications 41:439--459

\bibitem[{K{\i}l{\i}n{\c{c}}-Karzan and Y{\i}ld{\i}z(2015)}]{kilincc2015two}
K{\i}l{\i}n{\c{c}}-Karzan F, Y{\i}ld{\i}z S (2015) Two-term disjunctions on the
  second-order cone. Mathematical Programming 154:463--491

\bibitem[{Kolmogorov and Zabin(2004)}]{kolmogorov2004energy}
Kolmogorov V, Zabin R (2004) What energy functions can be minimized via graph
  cuts? IEEE Transactions on Pattern Analysis and Machine Intelligence
  26:147--159

\bibitem[{Lobo et~al(2007)Lobo, Fazel, and Boyd}]{Lobo2007}
Lobo MS, Fazel M, Boyd S (2007) Portfolio optimization with linear and fixed
  transaction costs. Annals of Operations Research 152:341--365

\bibitem[{Lov{\'a}sz(1983)}]{L:submodular-convex}
Lov{\'a}sz L (1983) Submodular functions and convexity. In: Bachem A, Korte B,
  Gr{\"o}tschel M (eds) Mathematical Programming The State of the Art: Bonn
  1982, Springer, Berlin, pp 235--257

\bibitem[{Luedtke et~al(2012)Luedtke, Namazifar, and Linderoth}]{Luedtke2012}
Luedtke J, Namazifar M, Linderoth J (2012) Some results on the strength of
  relaxations of multilinear functions. Mathematical Programming 136:325--351

\bibitem[{Luedtke et~al(2018)Luedtke, D'Ambrosio, Linderoth, and
  Schweiger}]{luedtke2018strong}
Luedtke J, D'Ambrosio C, Linderoth J, Schweiger J (2018) Strong convex
  nonlinear relaxations of the pooling problem. arXiv preprint arXiv:180302955

\bibitem[{Luk and Pagano(1980)}]{qp-m}
Luk FT, Pagano M (1980) Quadratic programming with {M}-matrices. Linear Algebra
  and its Applications 33:15--40

\bibitem[{Mahajan et~al(2017)Mahajan, Leyffer, Linderoth, Luedtke, and
  Munson}]{Mahajan2017}
Mahajan A, Leyffer S, Linderoth J, Luedtke J, Munson T (2017) Minotaur: A
  mixed-integer nonlinear optimization toolkit. {ANL/MCS-P8010-0817, Argonne
  National Lab}

\bibitem[{Modaresi et~al(2016)Modaresi, K{\i}l{\i}n{\c{c}}, and
  Vielma}]{modaresi2016intersection}
Modaresi S, K{\i}l{\i}n{\c{c}} MR, Vielma JP (2016) Intersection cuts for
  nonlinear integer programming: Convexification techniques for structured
  sets. Mathematical Programming 155:575--611

\bibitem[{Nemhauser et~al(1978)Nemhauser, Wolsey, and Fisher}]{Nemhauser1978}
Nemhauser GL, Wolsey LA, Fisher ML (1978) {An analysis of approximations for
  maximizing submodular set functions {I}}. Mathematical Programming
  14:265--294

\bibitem[{Orlin(2009)}]{Orlin2009}
Orlin JB (2009) A faster strongly polynomial time algorithm for submodular
  function minimization. Mathematical Programming 118:237--251

\bibitem[{Picard and Ratliff(1975)}]{picard1975minimum}
Picard JC, Ratliff HD (1975) Minimum cuts and related problems. Networks
  5:357--370

\bibitem[{Plemmons(1977)}]{plemmons1977m}
Plemmons RJ (1977) {M-matrix characterizations. I -- nonsingular M-matrices}.
  Linear Algebra and its Applications 18:175--188

\bibitem[{Poljak and Wolkowicz(1995)}]{poljak1995convex}
Poljak S, Wolkowicz H (1995) Convex relaxations of (0,1)-quadratic programming.
  Mathematics of Operations Research 20:550--561

\bibitem[{Stubbs and Mehrotra(1999)}]{stubbs1999branch}
Stubbs RA, Mehrotra S (1999) A branch-and-cut method for 0-1 mixed convex
  programming. Mathematical programming 86:515--532

\bibitem[{Vielma(2018)}]{V:cayley}
Vielma JP (2018) Small and strong formulations for unions of convex sets from
  the cayley embedding. To appear in Mathematical Programming, arXiv preprint
  arXiv:1704.03954

\bibitem[{Wei et~al(2013)Wei, Sestok, and Oppenheim}]{Wei2013}
Wei D, Sestok CK, Oppenheim AV (2013) Sparse filter design under a quadratic
  constraint: Low-complexity algorithms. IEEE T Signal Proces 61:857--870

\bibitem[{Wu et~al(2017)Wu, Sun, Li, and Zheng}]{Wu2017}
Wu B, Sun X, Li D, Zheng X (2017) Quadratic convex reformulations for
  semicontinuous quadratic programming. SIAM Journal on Optimization
  27:1531--1553

\bibitem[{Young(1981)}]{Young81}
Young N (1981) The rate of convergence of a matrix power series. Linear Algebra
  and its Applications 35:261--278

\end{thebibliography}

\end{document}